\documentclass[preprint,11pt]{article}
\usepackage{amsfonts}
\usepackage{mathrsfs}
\usepackage{amssymb,amsfonts,amsmath}
\usepackage{graphics}
\usepackage{esint}
\usepackage{epstopdf}
\graphicspath{{./figures/}}
 \usepackage{float}
\usepackage{times}
\usepackage{graphicx,multicol,enumerate,subfigure}
\usepackage{color}
\usepackage{fullpage}
\usepackage{setspace}
\usepackage[english]{babel}
\usepackage[version=3]{mhchem}
\usepackage{amsthm}
\usepackage{verbatim, latexsym}
\usepackage{colortbl}
\usepackage{multirow}
\usepackage{marginnote}
\usepackage{comment}
\usepackage{amstext}
\usepackage{bm}
\usepackage{wrapfig}
\usepackage{enumitem}

\usepackage{amssymb}
\usepackage{comment}
\usepackage{color}
\definecolor{black}{rgb}{0,0,0}

\definecolor{red}{rgb}{1,0,0}

\definecolor{blue}{rgb}{0,0,1}

\makeatletter
\theoremstyle{plain}
\newtheorem{thm}{\protect\theoremname}
  \theoremstyle{plain}
  \newtheorem{lem}[thm]{\protect\lemmaname}

\usepackage{float}

\newcommand{\norm}[1]{\left\|#1\right\|}

\newtheorem{theorem}{Theorem}

\newtheorem{definition}{Definition}

\usepackage{amsfonts} 
\makeatother

\usepackage{babel}

\title{Space-time Non-local multi-continua upscaling for parabolic equations with moving channelized media}
\author{Jiuhua Hu\thanks{Department of Mathematics, Texas A\&M University, College Station, TX 77843, USA.  (\texttt{E-mail: jiuhuahu@tamu.edu})},
~Wing Tat Leung\thanks{Department of Mathematics, University of California, Irvine, USA.(\texttt{E-mail: wtleung@uci.edu)})}, 
~Eric Chung\thanks{Department of Mathematics, The Chinese University of Hong Kong, Shatin, Hong Kong.  (\texttt{E-mail: tschung@math.cuhk.edu.hk})},
~Yalchin Efendiev\thanks{Department of Mathematics and Institute for Scientific Computation (ISC), Texas A\&M University, College Station, TX 77843, USA.  (\texttt{E-mail: efendiev@math.tamu.edu})}
~ and ~Sai-Mang Pun\thanks{Department of Mathematics, Texas A\&M University, College Station, TX 77843, USA.  (\texttt{E-mail: smpun@math.tamu.edu})}}
\makeatother

\usepackage{babel}
  \providecommand{\lemmaname}{Lemma}
\providecommand{\theoremname}{Theorem}

\begin{document}
\maketitle
\begin{abstract}
In this paper, we consider a parabolic  problem with time-dependent heterogeneous coefficients. 
Many applied problems have coupled space and time heterogeneities. Their homogenization
or upscaling requires cell problems that are formulated in space-time representative volumes
for problems with scale separation. In problems without scale separation, local problems include
multiple macroscopic variables and oversampled local problems, where these macroscopic parameters
are computed. These approaches, called Non-local multi-continua,  are proposed 
for problems with complex spatial heterogeneities in a number of previous papers. 
In this paper, we extend this approach for space-time heterogeneities, by identifying 
macroscopic parameters in space-time regions.
Our proposed method space-time Non-local multi-continua (space-time NLMC)  is an efficient 
numerical solver to deal with time-dependent heterogeneous coefficients. It provides a flexible 
and systematic way to construct multiscale basis functions to approximate the solution.
These multiscale basis functions are constructed by solving a local energy minimization 
problems in the oversampled space-time regions such that these multiscale basis functions 
decay exponentially outside the oversampled domain. Unlike the classical time-stepping methods 
combined with full-discretization technique, our space-time NLMC efficiently constructs the 
multiscale basis functions in a space-time domain and can provide a computational 
savings compared to space-only approaches as we discuss in the paper.  
We present two numerical experiments, which show that the proposed approach can provide 
a good accuracy.
\end{abstract}

\section{Introduction}
A broad range of scientific and engineering problems, for example, composite materials, porous media, turbulent transport in high Reynolds number flows, involve highly varying and heterogeneous multiscale features. A direct numerical treatment of solving these problems is challenging since a fine mesh discretization is needed to capture the multiscale features and this will result in an expensive 
computational cost. There have been many existing multiscale model reduction techniques in the literature to deal with multiscale problems. These multiscale approaches include homogenization approaches, multiscale finite element methods (MsFEMs) 
\cite{MR1455261}, heterogeneous multiscale methods (HMMs) \cite{MR1979846}, variational multiscale methods 
\cite{MR1660141, li2017error}, flux norm approach \cite{MR2721592}, generalized multiscale finite element methods (GMsFEMs) 
\cite{AdaptiveGMsFEM2016, egh12} and localized orthogonal decomposition (LOD) 
\cite{engwer2019efficient,Peterseim2014}. 
 
Homogenization is a well-known upscaling method.  It constructs  homogenized equations whose coefficients depend only on the macroscopic variable. 
  The solutions to the homogenized equations can be solved using coarse mesh and  serves as an approximation to the exact solution in the homogenization limit.  On the other hand,
the main idea of MsFEM and similar methods, like GMsFEM, is to construct multiscale basis functions which capture the small scale information within each coarse grid. The small scale information of the coarse grids is then brought to the large scales. GMsFEM is designed to construct more basis functions for each coarse region.   It has been successfully applied in simulating multiscale problems in channelized permeability. This is mainly because the local problems can 
correctly identify the necessary channels without any geometry interpretation.
Constraint Energy Minimizing Generalized Multiscale Finite Element Method  (CEM-GMsFEM) shares some ideas of GMsFEM. It constructs multiscale basis functions by solving a minimization problems on oversampling domains. It can be shown that with an appropriate choice of oversampling layer,  the convergence of the method is independent of the contrast from the heterogeneities and the error linearly decreases with respect to coarse mesh size. These approaches have achieved great success in the efficient and accurate simulation of heterogeneous problems.

Parabolic initial-boundary value problems arise in many practical applications. In many of these
problems, the heterogeneities have a dynamic nature. For example, channel features (which
play an important role in identifying macroscopic variables) can change in time.
The classical numerical treatment includes full-discretization of space and time.
The standard discretization methods in time and space are based on time-stepping 
methods combined with some spatial discretization technique. It provides a good accuracy when 
solving many parabolic problems. See \cite{thomee1984galerkin} for more details.
However, due to this disparity of scales, the classical numerical treatment becomes 
prohibitively expensive and even intractable for many multiscale applications.

Some multiscale methods have been coupled with the full-discretization techniques 
to reduce the dimension. See
\cite{multiscaleFV_parabolic_2009, multiscale_parabolic_2007, HM_parabolic_hom_2007,homo_parabolic_2007,CEM_parabolic2019}.
For example, the work 
\cite{CEM_parabolic2019} first utilizes CEM-GMsFEM to construct spatial multiscale basis functions and then full-discretization technique
is used. Although these methods have been successfully applied to many problems, they suffer from the separation of the time and space discretizations and  can only be applied to problems with time-independent multiscale coefficients. A more efficient technique is needed to simulate problems with time-dependent multiscale coefficients.

Next, we discuss an advantage of using space-time methods in contrast to space-only approaches
for parabolic problems with time-dependent heterogeneities. When using spatial basis functions,
one needs to generate multiscale basis functions for each ``fine-grid'' time instant within
a coarse-grid time interval. Thus, the number of 
multiscale basis functions, that capture fine-grid dynamics, is very large. 
While using space-time approaches, one can reduce the 
coarse-grid degrees of freedom to a fewer basis functions as the dynamics of heterogeneities
are embedded into multiscale basis functions. As an example is a moving channel or channels
(characterized as high contrast inclusions connecting boundaries of coarse-grid block), 
which typically
requires many spatial basis functions to capture each fine-grid move of the channel, while
it needs a few multiscale basis functions if we use space-time approach.

In the paper, we will develop and analyze a novel multiscale method for parabolic problems with time-dependent multiscale coefficients.
Our approach is based on Non-local multi-continua (NLMC) upscaling method and space-time finite element method. 
We assume that one knows each separate channel within each space-time coarse block and follow a general procedure in \cite{chung2018non} to construct a multiscale basis functions.
NLMC identifies space-time multi-continua parameters and defines a piece-wise constant functions as  local auxiliary functions. Next, multiscale basis functions are sought in the oversampled region subject to a constraint that  the minimizer is orthogonal to the auxiliary space. These multiscale functions are shown to decay exponentially outside the corresponding local oversampling regions. This exponential decay property plays a vital role in the convergence of the proposed method and justifies the use of the local multiscale basis functions.
In this paper, we construct local space-time ansatz spaces to approximate the global space-time ansatz spaces.

The remainder of this paper is organized as follows. In Section \ref{sec:prob_setting_SP}, we introduce the parabolic model problem, standard space-time weak formulation and functional spaces that will be used in this work. We develop local and global NLMC upscaling method in Section \ref{sec:BasisConstruction_SP}. Convergence analysis of our proposed method is studied in Section \ref{sec:convergence_SP}. We present numerical experiments in Section \ref{sec:numerical_SP} to demonstrate the performance of our  proposed method.
Concluding remarks are drawn in Section \ref{sec:conclusion_SP}.
\section{Problem Setting}\label{sec:prob_setting_SP}
In this section, we present some preliminaries of the model problems and introduce the necessary notations. 
Our aim is to develop an efficient numerical upscaling method for parabolic problems with time-dependent heterogeneous coefficients. 
Let $\Omega \subset \mathbb{R}^d$ ($d \in \{ 2, 3\}$) be a bounded domain with a sufficiently smooth boundary $\partial\Omega$. Let $T>0$ be a given positive terminal time. 
We seek a function $u = u(t,x)$ such that it solves the following initial boundary value problem: 
\begin{equation}
\label{PDEModel}
\left\{
\begin{aligned}
\partial_{t}u(t,x)-\nabla\cdot (\kappa(t,x)\nabla u(t,x))&=f(t,x), & (t,x)\in(0,T)\times\Omega, \\
u(0,x)&=0, &x\in \Omega ,\\
u(t,x)&=0, &(t,x)\in  [0,T]\times\partial\Omega,
\end{aligned}
\right .
\end{equation}
where $\kappa = \kappa(t,x)$ is a high-contrast time-dependent permeability field and $f \in L^2(0,T; L^2(\Omega))$ is a source function. We assume that there exist two positive constants $\kappa_0$ and $\kappa_1$ such that $0< \kappa_0 \leq \kappa(t,x) \leq \kappa_1$ for any $(t,x) \in \Omega_T := [0,T] \times \Omega$.

In this work, we will mainly focus on the case when $\kappa$ is a so-called channelized-moving medium. 
In particular, we assume that $\kappa$ is a piecewise constant function such that 
$$ \kappa (x,t) = \left \{ 
\begin{array}{cl}
\kappa_m & \text{if} ~ (t,x) \in D_m, \\
\kappa_i & \text{if} ~ (t,x) \in D_{c,i},
\end{array} \right . $$
where $\kappa_m$ and $\kappa_i$ are two positive constants between $\kappa_0$ and $\kappa_1$ such that the ratio $\kappa_i / \kappa_m$ is very large. Here, the space-time domain $\Omega_T$ is divided into two non-overlapping sets of regions in $\mathbb{R}^{d+1}$ with
\[ 
\Omega_T=D_m\bigcup_{i=1}^{\mathcal{I}_c} D_{c,i}. 
\]
The set $D_m$ is called the matrix region of the coefficient $\kappa$; $ D_{c,i}$ is called the $i$-th channel of the coefficient $\kappa$ and $\mathcal{I}_c$ is the total number of channels in the coefficient $\kappa$. 
In practice, the space-time {\it volume} of the matrix $D_m$ is much larger than that of the channelized region $D_c := \bigcup_{i=1}^{\mathcal{I}_c} D_{c,i}$. 
\subsection{Space-time variational formulation and space-time discretization}
Let $\alpha=(\alpha_1,\alpha_2,\cdots,\alpha_d)$ be a multi-index with non-negative integers $\alpha_i$ for $i=1,2,\cdots, d$. We use $|\alpha|$ to denote the sum of its elements, that is, $|\alpha|=\sum_{i=1}^d \alpha_i$.
For non-negative integers $l$ and $k$, we define a Sobolev space on the space-time domain $\Omega_T$ as 
$ H^{l,k}(\Omega_T):=\{u\in L^2(\Omega_T):\partial^\alpha_x u\in L^2(\Omega_T) \text{ for all }\alpha \text{ with } 
0\leq |\alpha|\leq l, \text{ and } \partial^i_t u\in L^2(\Omega_T) \text{ for } i=0,1,\cdots,k\}$. 
Moreover, we define $H^{1,0}_0(\Omega_T):=\{u\in H^{1,0}(\Omega_T): u(t,x)=0 \text{ for } x\in \partial\Omega\}$ and 
$H^{1,1}_{0,0}(\Omega_T):=\{u\in H^{1,1}(\Omega_T): u(t,x)=0 \text{ for } x\in \partial\Omega,
\text{ and }u(0,x)=0 \text{ for } x\in \Omega \}$.
The weak space-time variational formulation of \eqref{PDEModel} reads as follows: find $u\in H^{1,1}_{0,0}(\Omega_T)$ such that 
\begin{equation}
b(u,v)+a(u,v)=(f,v)\; \forall v\in H^{1,0}_{0}(\Omega_T),
\label{variation_form}
\end{equation}
where $b(u,v)=\int_{\Omega_T} \partial_t uv$, $a(u,v)=\int_{\Omega_T} \kappa\nabla_xu\nabla_xv$ and 
$(f,v)=\int_{\Omega_T} fv$. 

To discretize the variational problem \eqref{variation_form}, let $\mathcal{T}_{H} $ be a partition of space domain $\Omega$ into non-overlapping shape-regular rectangular elements with maximal mesh size $H$. 
The time domain $(0,T]$ is partitioned into $\mathcal{T}_{\Delta t}=\{(t_{i},t_{i+1}]\}_{i=0}^{N_T-1}$ with the maximal temporal mesh size $\Delta t:=\text{max}_{0\leq i\leq N_T-1}\{t_{i+1}-t_{i}\}$.
A space-time coarse element $K^{(n,i)}$ is then defined by $(t_{n},t_{n+1}]\times K^{i}$
for $K^{i}\in\mathcal{T}_{H}$ and $(t_{n},t_{n+1}]\in \mathcal{T}_{\Delta T}$. Furthermore, let $\mathcal{T}_h$ be a refinement of $\mathcal{T}_H$ and $\mathcal{T}_{\delta t}$
a refinement of $\mathcal{T}_{\Delta T}$. 

For each coarse space element $K^i$, we define the oversampled region $K^{i}_{k_i} \subseteq \Omega$ by enlarging $K^i$ by $k_i \in \mathbb{N}$ layer(s), i.e., 
$$ K^{i}_0 := K_i, \quad K^i_{k_i} := \bigcup \{ K \in \mathcal{T}_H : K \cap K^i_{k_i -1} \neq \emptyset \} \quad \text{for } k_i = 1, 2, \cdots.$$
For simplicity, we denote $K^i_+$ a generic oversampling region related to the coarse element $K^i$ with a specific oversampling parameter $k_i$. See Figure \ref{fig:oversamp} for an illustration of $K^i_{1}$. 
For each space-time coarse element $K^{(n,i)}$, its oversampling region is defined as the region enlarging $K^i$ by some coarse spatial layers and some temporal layers. For example, letting $t_{n}^{-}=t_{\max\{n-M,0\}}$,  the oversampling region of $K^{(n,i)}$ with $N_s$ spatial and $M$ temporal oversampling layers is defined as $(t_{n}^{-},t_{n+1}]\times K_{N_s}^{i}$.
Similarly, we denote $K^{(n,i)}_+$ a generic oversampling region related to the coarse space-time element $K^{(n,i)}$. 
\begin{figure}[H]
		\centering
		\includegraphics[trim={0.1cm 1.8cm 0.2cm 1.5cm},clip,width=0.34 \textwidth]{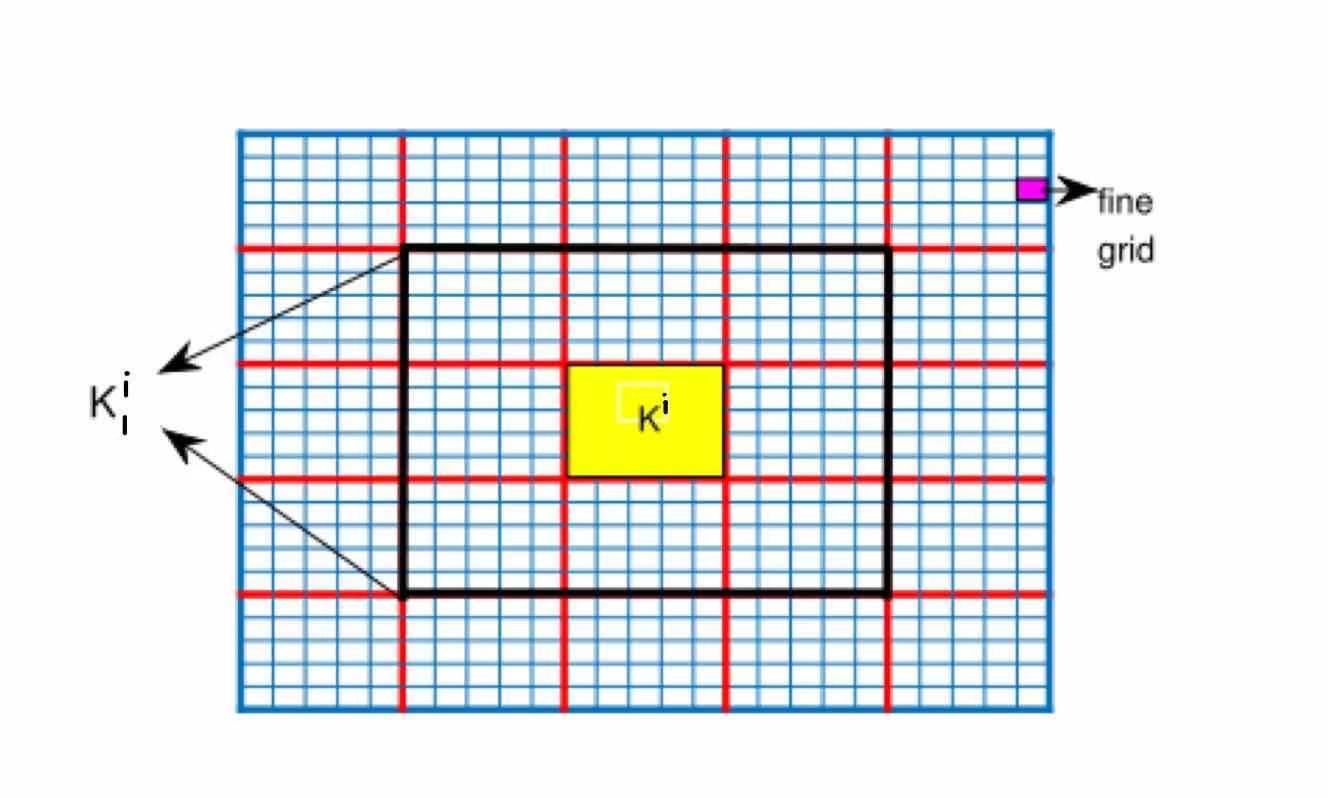}
		\caption{Illustration of oversampling space domain $K^i_1$.}
		\label{fig:oversamp}
\end{figure}

\subsection{Functional spaces and bilinear forms}
In this subsection, we introduce some functional spaces and bilinear forms used throughout the paper. For each $\omega\subset\Omega$ and $0\leq t_n<t_m\leq T$, we define the following functional spaces.
\begin{align*}
 V(t_{n},t_{m};\omega)&:=L^{2}(t_{n},t_{m};H^{1}(\omega))\cap H^{1}(t_{n},t_{m};L^{2}(\Omega)),\\
 W(t_{n},t_{m};\omega)&:=L^{2}(t_{n},t_{m};H^{1}(\omega)),\\
 V_{0}(t_{n},t_{m};\omega) & :=\{v\in V(t_{n},t_{m};\omega)|\;v(t_{n},\cdot)=0,\;v(t,x)=0\;\forall(t,x)\in(t_{n},t_{m})\times\partial\omega\},\\
 V_{d}(t_{n},t_{m};\omega) & :=\{v\in L^{2}(t_{n},t_{m};H^{1}(\omega))|\; v|_{(t_{k},t_{k+1})\times\omega}\in H^{1}(t_{k},t_{k+1};L^{2}(\Omega)),\; \forall k: n\leq k\leq m-1,\\
  &~~~~~~~~\text{  and } v(t,x)=0\;\forall(t,x)\in(t_{n},t_{m})\times\partial\omega\}.
\end{align*}

To shorten notations, we use $V$, $V_{0}$, $V_{1}$ and $W$  to denote $V(0,T;\Omega)$,
$V_{0}(0,T;\Omega)$, $V_{1}(0,T;\Omega)$, 
respectively.

Next, we will introduce some auxiliary functions $\psi_{j}^{(n,i)}$'s 
corresponding to different continua of the problem. Consider an oversampling region $K^{(n,i)}_{+}$ of the coarse space-time block $K^{(n,i)}$.
For any coarse space-time block $K^{(n',i')}\subset K^{(n,i)}_{+}$, we denote $F^{(n',i')}=\{f^{(n',i')}_k|f^{(n',i')}_k=D_c\cap K^{(n',i')}\neq \emptyset \}$
as a set containing discrete channels inside $K^{(n',i')}$. Set $L^{(n',i')}=|F^{(n',i')}|$. The functions $\psi_{j}^{(n,i)}$ for $j=0,1,\cdots, L^{(n,i)}$ are defined as follows:
\[
\int_{K^(n',i')}\tilde{\kappa} \psi_{0}^{(n,i)}= \delta_{n,n'}\delta_{i,i'} \text{ and } \int_{f^{(n',i')}_k}\tilde{\kappa} \psi_{\ell}^{(n,i)}= \delta_{n,n'}\delta_{i,i'} \delta_{\ell,k}.
\]
\\
We notice that
$\psi_{j}^{(n,i)}$ is supported in $K^{(n,i)}$. Let $V_{aux}^{(n,i)}=\text{span}_j\{\psi_j^{(n,i)}\}$ be the local auxiliary multiscale space corresponding to the coarse space-time block $K^{(n,i)}$. For any $\omega\in \Omega$ and $0\leq t_n< t_m\leq T$, we denote $V_{aux}(t_n,t_m;\omega)=\bigcup\{V_{aux}^{(k,i)}: K^i\subset\omega \text{ and } (t_k, t_{k+1})\subset (t_n,t_m)\}$. For simplicity, we shall use $V_{aux}$ to denote $V_{aux}(0,T;\Omega)$. We denote $N:=\dim(V_{aux} )$.

We now define $s(\cdot,\cdot)$ as a weighted $L^{2}$ inner production with
weighting function $\tilde{\kappa}:=\sum_{j} \kappa |\nabla \chi_j |^2$, that is
\[
s(u,v)=\int_{0}^{T}\int_{\Omega}\tilde{\kappa}uv.
\]
 Here, $\chi_j$'s are  the standard multiscale basis functions
defined coarse elementwise. On each coarse element $K\in \mathcal{T}_{H}$, it satisfies
\begin{align*} 
-\nabla\cdot(\kappa(x)\nabla\chi_j) &= 0  &&\quad\text{ in }\;\;K, \\
\chi_j &= g_j &&\quad\text{ on }\partial K, \nonumber
\end{align*}
where $g_j$ is affine over $\partial K$ with $g_j(O_i)=\delta_{ij}$ for all $i,j=1,\cdots, N$ and $\{O_i\}_{i=1}^{N}$ are the set of coarse nodes on $\mathcal{T}_{H}$. By its definition, $\chi_j$ is locally supported, that is,
$\text{supp}(\chi_j)\subset\omega_i.$

 Next, for each $\omega\subset\Omega$ and $0\leq t_{n}\leq t_{m}\leq T$,
we will define several bilinear operators $a(t_{n},t_{m};\omega;\cdot,\cdot):V(t_{n},t_{m};\omega)\times W(t_{n},t_{m};\omega)\rightarrow\mathbb{R}$,
$b(t_{n},t_{m};\omega;\cdot,\cdot):V(t_{n},t_{m};\omega)\times W(t_{n},t_{m};\omega)\rightarrow\mathbb{R}$
and $e(t_{n},t_{m};\omega;\cdot,\cdot):V(t_{n},t_{m};\omega)\times V(t_{n},t_{m};\omega)\rightarrow\mathbb{R}$
such that
\[
a(t_{n},t_{m};\omega;v,w)=\int_{t_{n}}^{t_{m}}\int_{\omega}\kappa\nabla v\cdot\nabla w,
\]
\[
e(t_{n},t_{m};\omega;v,w)=\sum_{i=n}^{m-1}\int_{t_{i}}^{t_{i+1}}\int_{\omega}\tilde{\kappa}^{-1}\partial_{t}v\partial_{t}w,
\]
\[
b(t_{n},t_{m};\omega;v,w)=\int_{t_{n}}^{t_{m}}\int_{\omega}\partial_t v w,
\]
and
\[
s(t_{n},t_{m};\omega;u,v)=\int_{t_n}^{t_m}\int_{\omega}\tilde{\kappa}uv.
\]
Then we will define $c(t_{n},t_{m};\omega;\cdot,\cdot)$ and $d(t_{n},t_{m};\omega;\cdot,\cdot)$
as
\[
c(t_{n},t_{m};\omega;v,w)=b(t_{n},t_{m};\omega;v,w)+a(t_{n},t_{m};\omega;v,w),
\]
 and
\[
d(t_{n},t_{m};\omega;v,w)=c(t_{n},t_{m};\omega;v,w)+e(t_{n},t_{m};\omega;v,w).
\]

Furthermore, we can define several norms related to the above bilinear operators.
For any $\omega\subset \Omega$ and $0\leq t_n<t_m\leq T$, we define:
\[
\|v\|_{L^{2}(\kappa,\omega)}^{2}=\int_{\omega} \kappa v^2,
\]
\[
\|v\|_{s(t_n,t_m;\omega)}^2=\int_{t_n}^{t_m} \int_\omega \tilde{\kappa} v^2,
\]
\[
\|v\|_{V(t_n,t_m;\omega)}^{2}=\int_{t_n}^{t_m}\|\nabla v\|_{L^{2}(\kappa,\omega)}^{2}+\int_{t_n}^{t_m}\|\tilde{\kappa}^{-\frac{1}{2}}\partial_{t}v\|_{L^{2}(\omega)}^2
\]
and
\[
\|v\|_{W(t_n,t_m;\omega)}^{2}=\int_{t_n}^{t_m}\int_{\omega}\kappa|\nabla v|^{2}.
\]
 We note that $\|v\|_{V(t_n,t_m;\omega)}^{2}=\|v\|_{W(t_n,t_m;\omega)}^{2}+\int_{t_n}^{t_m}\|\tilde{\kappa}^{-\frac{1}{2}}\partial_{t}v\|_{L^{2}(\omega)}^2$.
To simplify the notations, we denote $ \|v\|_{L^2(\omega)}=\|v\|_{L^2(0,T;\omega)}$, $ \|v\|_{s(\omega)}=\|v\|_{s(0,T;\omega)}$, $ \|v\|_{W(\omega)}=\|v\|_{W(0,T;\omega)}$, $ \|v\|_{V(\omega)}=\|v\|_{V(0,T;\omega)}$, 
$ \|v\|_{L^2(t_n,t_m)}=\|v\|_{L^2(t_n,t_m;\Omega)}$, $ \|v\|_{s(t_n,t_m)}=\|v\|_{s(t_n,t_m;\Omega)}$, $ \|v\|_{W(t_n,t_m)}=\|v\|_{W(t_n,t_m;\Omega)}$, $ \|v\|_{V(t_n,t_m)}=\|v\|_{V(t_n,t_m;\Omega)}$, 
$ \|v\|_{L^2}=\|v\|_{L^2(0,T;\Omega)}$, $ \|v\|_{s}=\|v\|_{s(0,T;\Omega)}$, $ \|v\|_{W}=\|v\|_{W(0,T;\Omega)}$ and $ \|v\|_{V}=\|v\|_{V(0,T;\Omega)}$.

\section{Space-time NLMC} \label{sec:BasisConstruction_SP}
In this section, we present the space-time NLMC upscaling method. First,  we construct global and the localized
space-time downscale operators, which can be used to define global space-time multiscale basis functions and local space-time multiscale basis functions. Then, we present the formulation of the coarse-grid solution.

\subsection{Global multiscale space}
We present the construction of the global downscale operator and the corresponding global numerical
solution.
We  define the global downscale operator $F:\mathbb{R}^{N}\rightarrow V_{0}\times V_{aux}$
by $U \mapsto (F_1(U),F_2(U))$ and
\begin{align*}
d(0,T;\Omega;F_{1}(U),w)-s(F_{2}(U),w) & =0,&\forall w\in V_{d},\\
s(F_{1}(U),\psi_{j}^{(n,i)}) & =U_{j}^{(n,i)},&\forall\psi_{j}^{(n,i)}\in V_{aux}.
\end{align*}
We remark here that the global downscale operator also defines the global basis functions.
Next, we can define the
global coarse grid problem as: finding $U\in\mathbb{R}^{N}$ such
that

\begin{equation}
  s(F_{2}(U),\psi_{j}^{(n,i)})=\int_{0}^{T}\int_{\Omega}f\psi_{j}^{(n,i)},\;\forall\psi_{j}^{(n,i)}\in V_{aux}
\label{gol_weak_form}  
\end{equation}
and the global numerical solution $u_{glo}$ is defined by 
\begin{equation}
   u_{glo}:=F_{1}(U). 
   \label{gol_sol}
\end{equation}

\subsection{Localization of global multiscale basis functions}

In this subsection, we will introduce the localized downscale operator
$F_{ms}=(F_{ms,1},F_{ms,2})$ and the localized coarse grid problem. For each space-time coarse block $K^{(n,i)}$ and its oversampled region $K_{+}^{(n,i)}=[t_n^{-},t_{n+1}]\times K^i_{+}$,
we  define a local downscale operator $F_{loc}^{(n,i)}:\mathbb{R}^{N}\rightarrow V_{0}(t_{n}^{-},t_{n+1};K_{+}^{i})\times V_{aux}(t_{n}^{-},t_{n+1};K_{+}^{i})$
by $U \mapsto (F_{loc,1}(U),F_{loc,2}(U))$ and
\begin{align*}
d(t_{n}^{-},t_{n+1};K_{+}^{i};F_{loc(U),1}^{(n,i)},w)-s(F_{loc,2}^{(n,i)}(U),w) & =0,&\forall w\in V_{d}(t_{n}^{-},t_{n+1};K_{+}^{i}),\\
s(F_{loc,1}^{(n,i)}(U),\psi_{j}^{(m,l)}) & =U_{j}^{(m,l)},&\forall\psi_{j}^{(m,l)}\in V_{aux}(t_{n}^{-},t_{n+1},K_{+}^{i}).
\end{align*}
Then the localized
downscale operator is defined by $F_{ms}(U)=\sum_{n,i}\chi^{(n,i)}F_{loc}^{(n,i)}(U)$
where $\chi^{(n,i)}$ is a partition of unity such that $\sum_{n,i}\chi^{(n,i)}\equiv1.$

The downscale operator also defines multiscale basis functions with support being $K_{+}^{(n,i)}$.
The coarse grid problem is then defined as: finding $U\in\mathbb{R}^{N}$
such that

\[
s(F_{ms,2}(U),\psi_{j}^{(n,i)})=\int_{0}^{T}\int_{\Omega}f\psi_{j}^{(n,i)},\;~~~\forall\psi_{j}^{(n,i)}\in V_{aux}
\]
and the localized numerical solution $u_{ms}$ is defined by 
\begin{equation}
    u_{ms}:=F_{ms,1}(U).
    \label{localized_sol}
\end{equation}
\section{Convergence Analysis} \label{sec:convergence_SP}

In this section, we will present a convergence analysis of the proposed method. We first prove in Theorem \ref{thm:glo_err} that the global numerical solution is a good approximation of the solution. 
Then we prove that the global downscale operators have a decay property with respect to the temporal oversampling layers and the local downscale operators have a decay property with respect to the spatial oversampling layers in Lemma \ref{lemma:decay_time} and Lemma \ref{lemma:decay_space}, respectively.
In this paper, we write $a \lesssim b$ if there exists a generic constant $C>0$ such that $ a \leq Cb$.

We first define a projection operator $\pi:L^{2}(0,T;L^{2}(\Omega))\rightarrow V_{aux}$
such that
\[
s(\pi(v),w)=s(v,w)\;\forall w\in V_{aux}.
\]
\textbf{Remark: }
It is easy to prove that  there exists a constant $C_0$ such that for all $w\in V$ 
\begin{equation}
   \cfrac{\|w-\pi(w)\|_{s}}{\|w\|_{V}}\leq C_{0}\Big(1+(\frac{\Delta t}{H^{2}})^{\frac{1}{2}}\Big).
\label{Assumption1} 
\end{equation}
We present the following result of the projection operator $\pi$.
\begin{lem}
\label{Assumtion2}
Let $K\in \mathcal{T}_H$ be any coarse spatial element and $[t_{n-1},t_n]\subset [0,T]$. Then there exists a constant $C_s>0$ such that for all $v_{aux}\in V_{aux}(t_{n-1},t_{n};K)$,
 there exist $w\in V_{0}(t_{n-1},t_{n};K)$ satisfying
\[
\pi(w)=v_{aux},\;\|w\|_{V}\leq C_{s}\|v_{aux}\|_{s}.
\]
\end{lem}
Lemma \ref{Assumtion2} can be proved using a similar technique in Lemma 3.2 \cite{zhao2020analysis}. For brevity of this article, we omit the proof.

Next, we establish the following estimates for later use in the analysis.
\begin{lem}\label{lemma: ineq_norm}
For any $\omega\subset \Omega$ and $0\leq t_{n}< t_{m}\leq T$,
 the following inequalities hold for any $u\in V_{0}(t_{n},t_{m};\omega)$,
\begin{align}
\|u\|_{W(t_{n},t_{m};\omega)}^{2} & \leq c(t_{n},t_{m};\omega;u,u),\\
\|u\|_{V(t_{n},t_{m};\omega)}^{2} & \leq d(t_{n},t_{m};\omega;u,u).
\end{align}
\end{lem}

\begin{proof}
It follows from the definitions of bilinear operators and $u(t_n)=0$ that we have
\begin{align*}
c(t_{n},t_{m};\omega;u,u) & =\int_{t_{n}}^{t_{m}}\int_{\omega}\kappa\nabla u\cdot\nabla u+\int_{t_{n}}^{t_{m}}\int_{\omega}(\partial_{t}u)u\\
 & =\int_{t_{n}}^{t_{m}}\int_{\omega}\kappa\nabla u\cdot\nabla u+\cfrac{1}{2}\|u(t_{m})\|_{L^{2}(\omega)}^{2}\\
 & \geq\|u\|_{W(t_{n},t_{m};\omega)}^{2}
\end{align*}
and
\begin{align*}
e(t_{n},t_{m};\omega;u,u) & =\sum_{i=n}^{m-1}\int_{t_{i}}^{t_{i+1}}\int_{\omega}\tilde{\kappa}^{-1}\partial_{t}u\partial_{t}u\\
 & =\|\tilde{\kappa}^{-\frac{1}{2}}\partial_{t}u\|_{L^{2}(t_{n},t_{m};\omega)}^2.
\end{align*}
Therefore, we have
\[
\|u\|_{V(t_{n},t_{m};\omega)}^{2}\leq c(t_{n},t_{m};\omega;u,u)+e(t_{n},t_{m};\omega;u,u)=d(t_{n},t_{m};\omega;u,u),
\]
which proves the second inequality. This completes the proof.
\end{proof}
To prove the convergence result of the proposed method, we first show the convergence result of using the global multiscale basis functions.
\begin{theorem}
\label{thm:glo_err}
Let $u$ be the exact solution of  \eqref{PDEModel} and $u_{glo}$ be the solution of \eqref{gol_sol}. We have
\[
\|u-u_{glo}\|_{V}\leq C_{0}(1+(\frac{\Delta t}{H^{2}})^{\frac{1}{2}})\|\tilde{\kappa}^{-\frac{1}{2}}f\|_{L^{2}}+\|\tilde{\kappa}^{-\frac{1}{2}}\partial_{t}u\|_{L^{2}}.
\]
Moreover, if the multiscale partition of unity $\chi_{i}$ is replaced by the bilinear partition of unity in the definition of $\tilde{\kappa}$, we have 
\[
\|u-u_{glo}\|_{V}\lesssim C_{0}(H+(\Delta t)^{\frac{1}{2}})\|\kappa^{-\frac{1}{2}}f\|_{L^{2}}+H\|\kappa^{-\frac{1}{2}}\partial_{t}u\|_{L^{2}}.
\]
\end{theorem}

\begin{proof}
Set $\tilde{f}:=\cfrac{f}{\tilde{\kappa}}$. It follows directly from \eqref{gol_weak_form} that we have
\[
s(F_{2}(U),v)=s(\tilde{f},v),\;\forall v\in V_{aux}.
\]
Therefore,
we have
\[
F_{2}(U)=\pi(\tilde{f}).
\]
We have
\begin{align*}
d(0,T;\Omega;u_{glo},v) & =s(F_{2}(U),v)\;\forall v\in W
\end{align*}
and
\[
c(0,T;\Omega;u,v)=s(\tilde{f},v)\;\forall v\in W.
\]
Then for any $v\in W$, the following equalities hold: 
\begin{align*}
d(0,T;\Omega;u-u_{glo},v) & =s(\tilde{f}-\pi(\tilde{f}),v)+e(0,T;\Omega;u,v)\\
 & =s(\tilde{f}-\pi(\tilde{f}),v-\pi(v))+e(0,T;\Omega;u,v)\\
 & =(f,v-\pi(v))+e(0,T;\Omega;u,v).
\end{align*}
Choosing $v=u-u_{glo}$ and utilizing Lemma \ref{lemma: ineq_norm} and \eqref{Assumption1}, we have
\begin{align*}
\|u-u_{glo}\|_{V} & \leq\|\tilde{\kappa}^{-\frac{1}{2}}f\|_{L^{2}}\cfrac{\|(u-u_{glo})-\pi(u-u_{glo})\|_{s}}{\|u-u_{glo}\|_{V}}+\|\tilde{\kappa}^{-\frac{1}{2}}\partial_{t}u\|_{L^{2}}\\
 & \leq C_{0}(1+(\frac{\Delta t}{H^{2}})^{\frac{1}{2}})\|\tilde{\kappa}^{-\frac{1}{2}}f\|_{L^{2}}+\|\tilde{\kappa}^{-\frac{1}{2}}\partial_{t}u\|_{L^{2}}.
\end{align*}
The second part of the theorem follows from the definition of $\tilde{\kappa}$ and $|\nabla\chi_{i}|=O(H^{-1})$.
\end{proof}
Theorem \ref{thm:glo_err} justifies the use of global downscale operators. Moreover, it also implies that the coarse time step size should be at most $O(H^2)$ to ensure a good accuracy. To prove our main theorem, we need two important lemmas.
We first show in Lemma \ref{lemma:decay_time} that the global downscale operators have a decay property with respect to the temporal oversampling layers. Then we prove in Lemma \ref{lemma:decay_space} that the local downscale operators have a decay property with respect to the spatial oversampling layers. Our main theorem shall be presented in Theorem \ref{thm:main_sp}. We first prove the following lemma, which will be frequently used in proofs.
\begin{lem}
\label{lemma2}
For any $K\in\mathcal{T}_{H}$ and $\Delta t\leq t_n\leq T$, if 
$v_{1}\in V(t_{n-1},t_{n};K)$ and $v_2\in V_{aux}(t_{n-1},t_{n};K)$
satisfy
\[
d(t_{n-1},t_{n};K;v_{1},w)=s(v_{2},w)\;\forall w\in V_{d}(t_{n-1},t_{n};K),
\]
then we have
\[
\|v_{2}\|_{s(t_{n-1},t_{n};K)}\leq\sigma\|v_{1}\|_{V(t_{n-1},t_{n};K)},
\]
where $\sigma=\Big(C_{s}\Big(C_{0}\Big(1+(\frac{\Delta t}{H^{2}})^{\frac{1}{2}}\Big)+1\Big)+1\Big).$
\end{lem}

\begin{proof}
It follows from Lemma \ref{Assumtion2} that there exists $w\in V_{0}(t_{n-1},t_{n};K)$ such that
\begin{equation*}
 \pi(w)=v_{2},\;\|w\|_{V}\leq C_s\|v_{2}\|_{s}.
\end{equation*}

Therefore, we have
\begin{equation}
 \begin{aligned}
\|v_{2}\|_{s}^{2} & = s(v_{2},w)\\
&=d(t_{n-1},t_{n};K;v_{1},w)\\
 & \leq\int_{t_{n-1}}^{t_{n}}\int_{K}\partial_{t}v_{1}w+\int_{t_{n-1}}^{t_{n}}\int_{K}\kappa\nabla v_{1}\cdot\nabla w+\int_{t_{n-1}}^{t_{n}}\int_{K}\tilde{\kappa}^{-1}\partial_{t}v_{1}\partial_{t}w\\
 & \leq\|v_{1}\|_{V(t_{n-1},t_{n};K)}\Big(\|w\|_{s(t_{n-1},t_{n};K)}+\|w\|_{V(t_{n-1},t_{n};K)}\Big).
\end{aligned}
\label{lemma:ineq1}
\end{equation}
Notice that 
\begin{equation*}
    \begin{aligned}
    \|w\|_{s(t_{n-1},t_{n};K)}&=\|w-\pi(w)\|_{s(t_{n-1},t_{n};K)}+\|v_{2}\|_{s(t_{n-1},t_{n};K)}\\
    &\leq C_{0}\Big(1+(\frac{\Delta t}{H^{2}})^{\frac{1}{2}}\Big)\|w\|_{V(t_{n-1},t_{n};K)}+\|v_{2}\|_{s(t_{n-1},t_{n};K)}.
    \end{aligned}
\end{equation*}
Then the following inequalities hold true:
\begin{equation}
 \begin{aligned}
\Big(\|w\|_{s(t_{n-1},t_{n};K)}+\|w\|_{V(t_{n-1},t_{n};K)}\Big) & \leq\Big(C_{0}\Big(1+(\frac{\Delta t}{H^{2}})^{\frac{1}{2}}\Big)+1\Big)\|w\|_{V(t_{n-1},t_{n};K)}+\|v_{2}\|_{s(t_{n-1},t_{n};K)}\Big)\\
 & \leq\Big(C_{s}\Big(C_{0}\Big(1+(\frac{\Delta t}{H^{2}})^{\frac{1}{2}}\Big)+1\Big)+1\Big)\|v_{2}\|_{s(t_{n-1},t_{n};K)}
\end{aligned}
\label{lemma:ineq2}
\end{equation}
A combination of \eqref{lemma:ineq1} and \eqref{lemma:ineq2} completes the proof. 
\end{proof}

Before deriving the error between the global and localized downscale operators, we introduce some notions to be used in the analysis. We first define two cut-off functions: cut-off function in temporal variable $\chi_{k,m}(t)$ and cut-off function in spatial variable $\chi^s_{k,m}(x)$.
\begin{definition}
For two non-negative integers $k,m$ with $0\leq k<m$, 
\begin{itemize}
    \item the cut-off function in time $\chi_{k,m}(t)$ is  defined as
\[
\chi_{k,m}(t):=\begin{cases}
1, & if\;t>t_{m},\\
\cfrac{t-t_{k}}{t_{m}-t_{k}}, & if\;t_{k}\leq t\leq t_{m},\\
0, & if\;t\leq t_{k};
\end{cases}
\]
\item the cut-off function in space $\chi^s_{k,m}(x)$ is  defined as a smooth function such that 
\begin{enumerate}[label=(\alph*)]
    \item $\chi^s_{k,m}(x)\in [0,1]$,
    \item \[
\chi^s_{k,m}(x)=\begin{cases}
1, & on\; K_k,\\
0, & on\; K_m,
\end{cases}
\]
\item $|\nabla\chi^s_{k,m}|^{2}\leq C_{\chi}\sum_{i}|\nabla\chi_{i}|^{2}$ for some constant $C_{\chi}$.
\end{enumerate}
\end{itemize}
\end{definition}
Note that $\chi_{k,m}(t)\in [0,1]$.  To simply the notations, for $0\leq k\leq n$, we denote $\chi_{k}(t)=\chi_{n-k,n-k+1}(t)$. 

Next, we shall define a temporal localized downscale operator $\tilde{F}_{loc}=(\tilde{F}_{loc,1},\tilde{F}_{loc,2})$.
\begin{definition}
The temporal localized downscale operator $\tilde{F}_{loc}^{(n)}:\mathbb{R}^{N}\rightarrow V_{0}(t_{n}^{-},t_{n+1};\Omega)\times V_{aux}(t_{n}^{-},t_{n+1};\Omega)$ 
are defined by $U\mapsto (\tilde{F}_{loc,1}^{(n)},\tilde{F}_{loc,2}^{(n)})$ and
\begin{align*}
d(t_{n}^{-},t_{n+1};\Omega;\tilde{F}_{loc,1}^{(n)}(U),w)-s(\tilde{F}_{loc,2}^{(n)}(U),w) & =0,&\forall w\in V_{d}(t_{n}^{-},t_{n+1};\Omega),\\
s(\tilde{F}_{loc,1}^{(n)}(U),\psi_{j}^{(m,l)}) & =U_{j}^{(m,l)},&\forall\psi_{j}^{(m,l)}\in V_{aux}(t_{n}^{-},t_{n+1};\Omega).
\end{align*}
\end{definition}

We  prove in the following lemma that the global downscale operator has a decay property with respect to the temporal oversampling layers. This also implies that the global multiscale basis functions has a decay property with respect to the temporal oversampling layers.
\begin{lem}
\label{lemma:decay_time}
Let $M$ be the number of temporal oversampling layers. For any space-time element $K=K^{(n,i)}$, $t_{n}^{-}=t_{n-M}$ and $U\in \mathbb{R}^N$, we have
\[
\|F_{1}(U)-\tilde{F}_{loc,1}^{(n)}(U)\|_{V(t_{n},t_{n+1})}\lesssim (1+\tilde{E}^{-1})^{1-M}\Big(\|F_{1}(U)\|_{V(t_{n}^{-},t_{n-M+1})}^{2}+\|F_{1}(U)\|_{s(t_{n}^{-},t_{n-M+1})}^{2}\Big),
\]
where  $\tilde{E}=\cfrac{C_{0}}{2}\Big(\Big(1+\frac{1}{\Delta t\min\{\tilde{\kappa}\}}\Big)+\sigma^{2}\Big)\Big(1+(\frac{\Delta t}{H^{2}})^{\frac{1}{2}}\Big)$.
\end{lem}

\begin{proof}
First, since $U\in \mathbb{R}^N$ satisfies the following equalities:
\begin{align*}
d(0,T;\Omega;F_{1}(U),w)-s(F_{2}(U),w) & =0,&\forall w\in V_{d},\\
s(F_{1}(U),\psi_{j}^{(m,l)}) & =U_{j}^{(m,l)},&\forall\psi_{j}^{(m,l)}\in V_{aux},
\end{align*}
and $V_{d}(t_{n-M},t_{n+1})\subset V_{d}$, we
have
\begin{align*}
d(t_{n-M},t_{n+1};\Omega;F_{1}(U),w)-s(F_{2}(U),w) & =0,&\forall w\in V_{d}(t_{n-M},t_{n+1}),\\
s(F_{1}(U),\psi_{j}^{(m,l)}) & =U_{j}^{(m,l)},&\forall\psi_{j}^{(m,l)}\in V_{aux}(t_{n-M},t_{n+1}).
\end{align*}

We define $\tilde{\eta}:=F(U)-\tilde{F}_{loc}^{(n)}(U)$, $\tilde{\eta_1}:=F_1(U)-\tilde{F}_{loc,1}^{(n)}(U)$ 
and $\tilde{\eta_2}:=F_2(U)-\tilde{F}_{loc,2}^{(n)}(U)$. Then the following equalities hold:
\begin{align*}
d(t_{n-M},t_{n+1};\Omega;\tilde{\eta}_{1},w)-s(\tilde{\eta}_{2},w) & =0, &\forall w\in V_{d}(t_{n-M},t_{n+1}),\\
s(\tilde{\eta}_{1},\psi_{j}^{(m,l)}) & =0,&\forall\psi_{j}^{(m,l)}\in V_{aux}(t_{n-M},t_{n+1}).
\end{align*}
We will estimate $\|\tilde{\eta}\|_{V(t_{n},t_{n+1};\Omega)}^{2}$
in three steps.\\

\textbf{Step 1}: We will prove
\begin{equation}
\|\tilde{\eta}_{1}\|_{V(t_{n-k+1},t_{n+1})}^{2}\leq\tilde{E}\|\tilde{\eta}_{1}\|_{V(t_{n-k},t_{n-k+1})}^{2}\;\text{for }1\leq k\leq M-1.
\label{ineq:step1}
\end{equation}
Let $w=\chi_{k}\tilde{\eta}_{1}$, for $k\leq M-1$. Since $\tilde{\eta}_{1}\in V_0$ and $1-\chi_{k}(t)=0$ if $t\geq t_{n-k+1}$, then $w\in V_{d}(t_{n-M},t_{n+1})$.
Then we have
\begin{equation}
 \begin{aligned}
d(t_{n-M},t_{n+1};\Omega;\tilde{\eta}_{1},\chi_{k}\tilde{\eta}_{1})-s(\tilde{\eta}_{2},\chi_{k}\tilde{\eta}_{1}) & =0,\\
s(\tilde{\eta}_{1},\psi_{j}^{(m,l)}) & =0,~~~~~~~~~\;\forall\psi_{j}^{(m,l)}\in V_{aux}(t_{n-k},t_{n+1}).
\end{aligned}  
\label{ineq:step1_2}
\end{equation}
Notice that
\begin{align*}
\int_{t_{n-k}}^{t_{n+1}}\int_{\Omega}(\tilde{\eta}_{1})_{t}\chi_{k}\tilde{\eta}_{1} & =-\int_{t_{n-k}}^{t_{n+1}}\int_{\Omega}(\chi_{k}\tilde{\eta}_{1})_{t}\tilde{\eta}_{1}+\int_{\Omega}\tilde{\eta}_{1}^{2}(t_{n+1},\cdot)\\
 & =-\int_{t_{n-k}}^{t_{n+1}}\int_{\Omega}(\tilde{\eta}_{1})_{t}\chi_{k}\tilde{\eta}_{1}-\cfrac{1}{\Delta t}\int_{t_{n-k}}^{t_{n-k+1}}\int_{\Omega}\tilde{\eta}_{1}^{2}+\int_{\Omega}\tilde{\eta}_{1}^{2}(t_{n+1},\cdot). 
\end{align*}
This gives 
\begin{equation}
    \int_{t_{n-k}}^{t_{n+1}}\int_{\Omega}(\tilde{\eta}_{1})_{t}\chi_{k}\tilde{\eta}_{1}= -\cfrac{1}{2\Delta t}\int_{t_{n-k}}^{t_{n-k+1}}\int_{\Omega}\tilde{\eta}_{1}^{2}+\cfrac{1}{2} \int_{\Omega}\tilde{\eta}_{1}^{2}(t_{n+1},\cdot). 
    \label{ineq:step1_3}
\end{equation}
Combining \eqref{ineq:step1_2} and \eqref{ineq:step1_3}, we arrive at the following estimate:
 \begin{equation}
\begin{aligned}
\|\eta_{1}\|_{V(t_{n-k+1},t_{n+1};\Omega)}^{2}+\cfrac{1}{2}\int_{\Omega}\tilde{\eta}_{1}^{2}(t_{n+1},\cdot) & \leq\cfrac{1}{2\Delta t}\int_{t_{n-k}}^{t_{n-k+1}}\int_{\Omega}\tilde{\eta}_{1}^{2}+s(\tilde{\eta}_{2},\chi_{k}\tilde{\eta}_{1})\\
 & =\cfrac{1}{2\Delta t}\int_{t_{n-k}}^{t_{n-k+1}}\int_{\Omega}\tilde{\eta}_{1}^{2}+\int_{t_{n-k}}^{t_{n-k+1}}\int_{\Omega}\tilde{\kappa}\chi_{k}\tilde{\eta}_{2}\tilde{\eta}_{1}.
\end{aligned}
\label{ineq:step1_4}
 \end{equation}


Utilizing \eqref{ineq:step1_4} and Cauchy-Schwarz inequality, one can show that
\[
\|\tilde{\eta}_{1}\|_{V((t_{n-k+1},t_{n+1})}^{2}\leq\cfrac{1}{2}\Big((1+\frac{1}{\Delta t\min\{\tilde{\kappa}\}})\|\tilde{\eta}_{1}\|_{s((t_{n-k},t_{n-k+1})}^{2}+\|\tilde{\eta}_{2}^{2}\|_{s((t_{n-k},t_{n-k+1})}\Big).
\]
Since $d(t_{n-k},t_{n-k+1};\Omega;\tilde{\eta}_{1},w)=s(\tilde{\eta}_{2},w)$ for any $ w\in V_{d}(t_{n-k},t_{n-k+1})$,
it follows from Lemma \ref{lemma2} that
\[
\|\tilde{\eta}_{2}^{2}\|_{s(t_{n-k},t_{n-k+1})}\leq\sigma^{2}\|\tilde{\eta}_{1}\|_{s(t_{n-k},t_{n-k+1})}^{2}.
\]
Therefore, we have
\[
\|\tilde{\eta}_{1}\|_{V(t_{n-k+1},t_{n+1})}^{2}\leq\tilde{E}\|\tilde{\eta}_{1}\|_{V(t_{n-k},t_{n-k+1})}^{2},
\]
where $\tilde{E}=\cfrac{C_{0}}{2}\Big(\Big(1+\frac{1}{\Delta t\min\{\tilde{\kappa}\}}\Big)+\sigma^{2}\Big)\Big(1+(\frac{\Delta t}{H^{2}})^{\frac{1}{2}}\Big)$.\\

\textbf{Step 2}: We will prove
\begin{equation}
  \|\tilde{\eta}_{1}\|_{V(t_{n},t_{n+1})}^{2}\leq(1+\tilde{E}^{-1})^{1-M}\|\tilde{\eta}_{1}\|_{(t_{n-M+1},t_{n+1})}^{2}.  
  \label{lemma_temp_decay_step2}
\end{equation}

Using Inequality \eqref{ineq:step1}, we have the following estimate: for $1\leq k\leq M-1$, 
\begin{align*}
\|\tilde{\eta}_{1}\|_{V(t_{n-k},t_{n+1})}^{2} & =\|\tilde{\eta}_{1}\|_{V(t_{n-k+1},t_{n+1})}^{2}+\|\tilde{\eta}_{1}\|_{V(t_{n-k},t_{n-k+1})}^{2}\\
 & \geq(1+\tilde{E}^{-1})\|\tilde{\eta}_{1}\|_{V(t_{n-k+1},t_{n+1})}^{2},
\end{align*}
Using the above inequality recursively,  we obtain \eqref{lemma_temp_decay_step2}.\\

\textbf{Step 3}: We will prove
\[
\|\tilde{\eta}_{1}\|_{V(t_{n-M+1},t_{n+1})}^{2}\leq\Big(1+\cfrac{1}{\Delta t\min\{\tilde{\kappa}\}}\Big)\|F_{1}(U)\|_{s(t_{n-M},t_{n-M+1})}^{2}+\|F_{1}(U)\|_{V(t_{n-M},t_{n-M+1})}^{2}.
\]

Since $\tilde{\eta}_{1}\in V_{d}(t_{n-M},t_{n+1})$, the following equalities hold true:
\begin{align*}
d(t_{n-M},t_{n+1};\Omega;\tilde{\eta}_{1},\tilde{\eta}_{1}) & =d(t_{n-M},t_{n+1};\Omega;\tilde{\eta}_{1},\tilde{\eta}_{1})-s(\tilde{\eta}_{2},\tilde{\eta}_{1})\\
 & =0.
\end{align*}
Using a similar derivation to obtain \eqref{ineq:step1_3}, one can also show that
\begin{align*}
\cfrac{1}{2}\Big(\int_{\Omega}\tilde{\eta}_{1}^{2}(t_{n+1},\cdot)-\int_{\Omega}\tilde{\eta}_{1}^{2}(t_{n-M},\cdot)\Big)+\|\tilde{\eta}_{1}\|_{V(t_{n-M},t_{n+1})}^{2} & =d(t_{n-M},t_{n+1};\Omega;\tilde{\eta}_{1},\tilde{\eta}_{1})\\
 & =0.
\end{align*}
Notice that
\begin{equation}
  \int_\Omega \tilde{\eta}_{1}^{2}(t_{n-M},\cdot)=\int_\Omega F_{1}^{2}(t_{n-M},\cdot) = -2\int_{t_{n-M}}^{t_{n-M+1}}\int_{\Omega}\partial_{t}((1-\chi_{M})F_{1}(U))(1-\chi_{M})F_{1}(U). 
  \label{ineq:step3_1}
\end{equation}
Utilizing \eqref{ineq:step3_1}, we have
\begin{align*}
\|\tilde{\eta}_{1}\|_{V(t_{n-M},t_{n+1})}^{2} & \leq\cfrac{1}{2}\int_{\Omega}\tilde{\eta}_{1}^{2}(t_{n-M},\cdot)\\
&=-\int_{t_{n-M}}^{t_{n-M+1}}\int_{\Omega}\partial_{t}((1-\chi_{M})F_{1}(U))(1-\chi_{M})F_{1}(U)\\
 & \leq\cfrac{1}{\Delta t\min\{\tilde{\kappa}\}}\|F_{1}(U)\|_{s(t_{n-M},t_{n-M+1})}^{2}+\|F_{1}(U)\|_{V(t_{n-M},t_{n-M+1})}\|F_{1}(U)\|_{s(t_{n-M},t_{n-M+1})}\\
 & \leq\Big(1+\cfrac{1}{\Delta t\min\{\tilde{\kappa}\}}\Big)\|F_{1}(U)\|_{s(t_{n-M},t_{n-M+1})}^{2}+\|F_{1}(U)\|_{V(t_{n-M},t_{n-M+1})}^{2}.
\end{align*}
The proof is completed using $\|\tilde{\eta}_{1}\|_{V(t_{n-M},t_{n+1})}^{2}\geq \|\tilde{\eta}_{1}\|_{V(t_{n-M+1},t_{n+1})}^{2}$ together with Step 2.
\end{proof}
Define the constant 
\[
C_\kappa=\text{sup}_{v\in V} \cfrac{\|v\|_s }{\|v\|_V}.
\]
We now prove in the following lemma that local downscale operators have a decay property with respect to the spatial oversampling layers.
\begin{lem}
\label{lemma:decay_space}
Let $N_s$ be the number of the oversampling layers in space.
For any coarse space element $K=K^i\in \mathcal{T}_{H}$ and time element $[t_n,t_{n+1}]\in \mathcal{T}_{\Delta t}$, we have
\[
\|\tilde{F}_{loc,1}^{(n)}(U)-F_{loc,1}^{(n,i)}(U)\|_{V(t_{n}^{-},t_{n+1};K)}^{2}\leq C_{\chi}E^{1-N_s}\Big(\|\tilde{F}_{loc,1}^{(n)}(U)\|_{V(t_{n}^{-},t_{n+1};K_{N}\backslash K_{N-1})}^{2}+\|\tilde{F}_{loc,1}^{(n)}(U)\|_{s(t_{n}^{-},t_{n+1};K_{N}\backslash K_{N-1})}^{2}\Big),
\]
where $E:=1+\cfrac{1}{2C_{\chi}^2+\sigma C_\kappa}$.
\end{lem}

\begin{proof}
Notice that $V_{d}(t_{n}^{-},t_{n+1};K_{+}^{i})\subset V_{d}(t_{n}^{-},t_{n+1};\Omega) $. It follows from the definitions of $ \tilde{F}_{loc}$ and ${F}_{loc}^{(n,i)}$ that   the following equalities hold true:
\begin{equation}
    \begin{aligned}
d(t_{n}^{-},t_{n+1};K_{+}^{i};F_{loc,1}^{(n,i)}(U)-\tilde{F}_{loc,1}^{(n)}(U),v)+s(F_{loc,2}^{(n,i)}(U)-\tilde{F}_{loc,2}^{(n)}(U),v) & =0,&\forall v\in V_{d}(t_{n}^{-},t_{n+1};K_{+}^{i}),\\
s(F_{loc,1}^{(n,i)}(U)-\tilde{F}_{loc,1}^{(n)}(U),\psi_{j}^{(n,i)}) & =0,&\forall\psi_{j}^{(n,i)}\in V_{aux}(t_{n}^{-},t_{n+1};K_{+}^{i}).
\end{aligned}
\label{lemma_space:system1}
\end{equation}

In this proof, we denote $\|\cdot\|_{V(\omega)}$, $\|\cdot\|_{W(\omega)}$
and $\|\cdot\|_{s(\omega)}$ by
\[
\|v\|_{V(\omega)}:=\|v\|_{V(t_{n}^{-},t_{n+1};\omega)},\;\|v\|_{W(\omega)}:=\|v\|_{W(t_{n}^{-},t_{n+1};\omega)}\;,\|v\|_{s(\omega)}:=\|v\|_{s(t_{n}^{-},t_{n+1};\omega)}.
\]
We then define $\eta=F_{loc}^{(n,i)}(U)-\tilde{F}_{loc}^{(n)}(U)$ and  $\eta_{j}=F_{loc,j}^{(n,i)}(U)-\tilde{F}_{loc,j}^{(n)}(U)$ for $j=1,2$.
For $k=1,2,\cdots, N_s$, we denote $\chi_{k,k-1}:=1-\chi_{k-1,k}$. Then we have
\begin{equation}
    \begin{aligned}
 & \|\eta_{1}(t_{n+1})\|_{L^{2}(K_{k-1}^{})}^{2}+\|\eta_{1}\|_{V(K_{k-1}^{})}^{2}\\
\leq & d(t_{n}^{-},t_{n+1};K_{k}^{};\eta_{1},\chi_{k,k-1}\eta_{1})-a(t_{n}^{-},t_{n+1};K_{k}^{}\backslash K_{k-1}^{};\eta_{1},\chi_{k,k-1}\eta_{1}).
\end{aligned}
\label{lemma_space:ineq1}
\end{equation}
Notice that $\pi(\chi_{k,k-1}\eta_{1})|_{K_{k-1}^{}}=\pi(\eta_{1})|_{K_{k-1}^{}}=0$. Choosing $v=\chi_{k,k-1}\eta_{1}$ in \eqref{lemma_space:system1} and utilizing Cauchy-Schwartz Inequality, we have
\begin{align*}
d(t_{n}^{-},t_{n+1};K_{k}^{};\eta_{1},\chi_{k,k-1}\eta_{1}) & =-s(\eta_{2},\chi_{k,k-1}\eta_{1})\\
 & \leq\|\eta_{2}\|_{s(K_{k}^{}\backslash K_{k-1}^{})}\|\chi_{k,k-1}\eta_{1}\|_{s(K_{k}^{}\backslash K_{k-1}^{})}\\
 & \leq\|\eta_{2}\|_{s(K_{k}^{}\backslash K_{k-1}^{})}\|\eta_{1}\|_{s(K_{k}^{}\backslash K_{k-1}^{})}.
\end{align*}
It follows from Lemma  \ref{lemma2} that we have
\begin{align*}
\|\eta_{2}\|_{s(K_{k}^{}\backslash K_{k-1}^{})} & \leq\sigma\|\eta_{1}\|_{V(K_{k}^{}\backslash K_{k-1}^{})}.
\end{align*}
Moreover, we have $\|\eta_{1}\|_{s(K_{k}^{}\backslash K_{k-1}^{})}\leq C_{\kappa} \|\eta_{1}\|_{V(K_{k}^{}\backslash K_{k-1}^{})} $.
Therefore 
\begin{equation}
  d(t_{n}^{-},t_{n+1};K_{k}^{};\eta_{1},\chi_{k,k-1}\eta_{1})\leq \sigma C_{\kappa}  \|\eta_{1}\|^2_{V(K_{k}^{}\backslash K_{k-1}^{})}.
\label{lemma_space:ineq2}  
\end{equation}
Since $\nabla(\chi_{k,k-1}\eta_{1})=\eta_{1}\nabla(\chi_{k,k-1})+\chi_{k,k-1}\nabla(\eta_{1})$
and $|\nabla\chi_{k,k-1}|^{2}\leq C_{\chi}\sum_{i}|\nabla\chi_{i}|^{2}$,
we have
\begin{align*}
\int_{t_{n}^{-}}^{t_{n+1}}\int_{K_{k}^{}\backslash K_{k-1}^{}}\kappa\nabla\eta_{1}\cdot\nabla\Big(\chi_{k,k-1}\eta_{1}\Big) & \leq\|\eta_{1}\|_{W(K_{k}^{}\backslash K_{k-1}^{})}\|\chi_{k,k-1}\eta_{1}\|_{W(K_{k}^{}\backslash K_{k-1}^{})}\\
 & \leq C_{\chi}\|\eta_{1}\|_{W(K_{k}^{}\backslash K_{k-1}^{})}\Big(\|\eta_{1}\|_{W(K_{k}^{}\backslash K_{k-1}^{})}+\|\eta_{1}\|_{s(K_{k}^{}\backslash K_{k-1}^{})}\Big)\\
 & \leq C_{\chi}\|\eta_{1}\|_{W(K_{k}^{}\backslash K_{k-1}^{})}\Big(\|\eta_{1}\|_{W(K_{k}^{}\backslash K_{k-1}^{})}+\|\eta_{1}\|_{V(K_{k}^{}\backslash K_{k-1}^{})}\Big).
\end{align*}
Using $\|\eta_{1}\|_{W(K_{k}^{}\backslash K_{k-1}^{})}\leq \|\eta_{1}\|_{V(K_{k}^{}\backslash K_{k-1}^{})}$, we obtain
\begin{equation}
    \int_{t_{n}^{-}}^{t_{n+1}}\int_{K_{k}^{}\backslash K_{k-1}^{}}\kappa\nabla\eta_{1}\cdot\nabla\Big(\chi_{k,k-1}\eta_{1}\Big)\leq 2C_{\chi}\|\eta_{1}\|^2_{V(K_{k}^{}\backslash K_{k-1}^{})}.
    \label{lemma_space:ineq3}
\end{equation}
A combination of \eqref{lemma_space:ineq1}, \eqref{lemma_space:ineq2} and \eqref{lemma_space:ineq3}, we arrive at
\begin{align*}
\|\eta_{1}(t_{n+1})\|_{L^{2}(K_{k-1}^{})}^{2}+\|\eta_{1}\|_{V(K_{k-1}^{})}^{2} 
 & \leq (2C_{\chi}^2+\sigma C_\kappa)\|\eta_{1}\|_{V(K_{k}^{}\backslash K_{k-1}^{})}^{2}\\
 & =(2C_{\chi}^2+\sigma C_\kappa)\Big(\|\eta_{1}\|_{V(K_{k}^{})}^{2}-\|\eta_{1}\|_{V(K_{k-1}^{})}^{2}\Big),
\end{align*}
which gives

\[
\|\eta_{1}\|_{V(K_{k-1}^{})}^{2}\leq\Big(1+\cfrac{1}{2C_{\chi}^2+\sigma C_\kappa}\Big)^{-1}\|\eta_{1}\|_{V(K_{k}^{})}^{2}, \text{ for } 1\leq k\leq N_s-1.
\]
Denote $E:=1+\cfrac{1}{2C_{\chi}^2+\sigma C_\kappa}$. Using above Inequality recursively, we obtain
\[
\|\eta_{1}\|_{V(K)}^{2}\leq E^{1-N_s}\|\eta_{1}\|_{V(K_{N_s-1}^{})}^{2}.
\]
It remains to estimate $ \|\eta_{1}\|_{V(K_{N_s-1}^{})}^{2}$.
We shall prove:
\[
\|\eta_{1}\|_{V(K_{N_s-1}^{})}^{2}\leq C_{\chi}\Big(\|\tilde{F}_{loc,1}^{(n)}(U)\|_{V(t_{n}^{-},t_{n+1};K_{N_s}\backslash K_{N_s-1})}^{2}+\|\tilde{F}_{loc,1}^{(n)}(U)\|_{s(t_{n}^{-},t_{n+1};K_{N_s}\backslash K_{N_s-1})}^{2}\Big).
\]
Notice that
\begin{equation}
\begin{aligned}
&\cfrac{1}{2}\|\eta_{1}(t_{n+1})\|^{2}+\|\eta_{1}\|_{V(K_{N_s})}^{2}\\
 =& d(t_{n}^{-},t_{n+1};K_{N_s}^{};\eta_{1},\eta_{1})\\
 =& d(t_{n}^{-},t_{n+1};K_{N_s}^{};\eta_{1},F_{loc,1}^{(n,i)}(U)-\chi_{N_{s},N_{s}-1}\tilde{F}_{loc,1}^{(n)}(U))\\
  &~~~~~+d(t_{n}^{-},t_{n+1};K_{N_s}^{};\eta_{1},(\chi_{N_{s},N_{s}-1}-1)\tilde{F}_{loc,1}^{(n)}(U)).
\end{aligned} 
\label{lemma5:2terms}
\end{equation}
We next estimate each of the above two terms. Choosing $v=F_{loc,1}^{(n,i)}(U)-\chi_{N_{s},N_{s}-1}\tilde{F}_{loc,1}^{(n)}(U)$ in \eqref{lemma_space:system1} and using Cauchy-Schwartz Inequality, we have the following  estimate:
\begin{align*}
 & d(t_{n}^{-},t_{n+1};K_{N_s}^{};\eta_{1},F_{loc,1}^{(n,i)}(U)-\chi_{N_{s},N_{s}-1}\tilde{F}_{loc,1}^{(n)}(U))\\
= & -s(\eta_{2},F_{loc,1}^{(n,i)}(U)-\chi_{N_{s},N_{s}-1}\tilde{F}_{loc,1}^{(n)}(U))\\
= & -s(\eta_{2},(1-\chi_{N_{s},N_{s}-1})\tilde{F}_{loc,1}^{(n)}(U))\\
\leq & \|\eta_{2}\|_{s(K_{N_{s}}^{}\backslash K_{N_{s}-1}^{})}\|\tilde{F}_{loc,1}^{(n)}(U)\|_{s(K_{N_{s}}^{}\backslash K_{N_{s}-1}^{})}.
\end{align*}
Furthermore, since $\|\eta_{2}\|_{s(K_{N_{s}}^{}\backslash K_{N_{s}-1}^{})}\leq\sigma\|\eta_{1}\|_{V(K_{N_{s}}^{}\backslash K_{N_{s}-1}^{})}$,
we have the following estimate:
\begin{equation}
 \begin{aligned}
 & d(t_{n}^{-},t_{n+1};K_{N_s}^{};\eta_{1},F_{loc,1}^{(n,i)}(U)-\chi_{N_{s},N_{s}-1}\tilde{F}_{loc,1}^{(n)}(U))\\
\leq & \sigma\|\eta_{1}\|_{V(K_{N_{s}}^{}\backslash K_{N_{s}-1}^{})}\|\tilde{F}_{loc,1}^{(n)}(U)\|_{s(K_{N_{s}}^{}\backslash K_{N_{s}-1}^{})}.
\end{aligned}
\label{lemma5:estimate_1st}
\end{equation}
We also have 
\begin{align*}
 & d(t_{n}^{-},t_{n+1};K_{N_s}^{};\eta_{1},(\chi_{N_{s},N_{s}-1}-1)\tilde{F}_{loc,1}^{(n)}(U))\\
\leq & \|\tilde{\kappa}^{-\frac{1}{2}}\partial_{t}\eta_{1}\|_{L^{2}(K_{N_{s}}^{}\backslash K_{N_{s}-1}^{})}(\|\tilde{F}_{loc,1}^{(n)}(U)\|_{s(K_{N_{s}}^{}\backslash K_{N_{s}-1}^{})}+\|\tilde{\kappa}^{-\frac{1}{2}}\partial_{t}\tilde{F}_{loc,1}^{(n)}(U)\|_{L^{2}(K_{N_{s}}^{}\backslash K_{N_{s}-1}^{})})\\
 & +\|\eta_{1}\|_{W(K_{N_{s}}^{}\backslash K_{N_{s}-1}^{})}\|(\chi_{N_{s},N_{s}-1}-1)\tilde{F}_{loc,1}^{(n)}(U)\|_{W(K_{N_{s}}^{}\backslash K_{N_{s}-1}^{})}.
\end{align*}
Notice that 
\[
\|(\chi_{N_{s},N_{s}-1}-1)\tilde{F}_{loc,1}^{(n)}(U)\|_{W(K_{N_{s}}^{}\backslash K_{N_{s}-1}^{})}\leq C_{\chi}\Big(\|\tilde{F}_{loc,1}^{(n)}(U)\|_{W(K_{N_{s}}^{}\backslash K_{N_{s}-1}^{})}+\|\tilde{F}_{loc,1}^{(n)}(U)\|_{s(K_{N_{s}}^{}\backslash K_{N_{s}-1}^{})}\Big).
\]
We obtain the following estimate: 
\begin{equation}
 \begin{aligned}
&d(t_{n}^{-},t_{n+1};K_{N_s}^{};\eta_{1},(\chi_{N_{s},N_{s}-1}-1)\tilde{F}_{loc,1}^{(n)}(U))  \\
\leq & C_{\chi}\|\eta_{1}\|_{V(K_{N_{s}}^{}\backslash K_{N_{s}-1}^{})}\Big(\|\tilde{F}_{loc,1}^{(n)}(U)\|_{V(K_{N_{s}}^{}\backslash K_{N_{s}-1}^{})}+\|\tilde{F}_{loc,1}^{(n)}(U)\|_{s(K_{N_{s}}^{}\backslash K_{N_{s}-1}^{})}\Big).
\end{aligned}
\label{lemma5:estimate_2ndterm}
\end{equation}
Combing \eqref{lemma5:estimate_1st} and \eqref{lemma5:estimate_2ndterm}, we arrive at 
\[
\cfrac{1}{2}\|\eta_{1}(t_{n+1})\|^{2}+\|\eta_{1}\|_{V(K_{N_s})}^{2}\leq C_{\chi}\|\eta_{1}\|_{V(K_{N_{s}}^{}\backslash K_{N_{s}-1}^{})}\Big(\|\tilde{F}_{loc,1}^{(n)}(U)\|_{V(K_{N_{s}}^{}\backslash K_{N_{s}-1}^{})}+\|\tilde{F}_{loc,1}^{(n)}(U)\|_{s(K_{N_{s}}^{}\backslash K_{N_{s}-1}^{})}\Big).
\]
Therefore, 
\[
\|\eta_{1}\|_{V(K_{N_s})}^{2}\leq C_{\chi}\Big(\|\tilde{F}_{loc,1}^{(n)}(U)\|_{V(K_{N_{s}}^{}\backslash K_{N_{s}-1}^{})}^{2}+\|\tilde{F}_{loc,1}^{(n)}(U)\|_{s(K_{N_{s}}^{}\backslash K_{N_{s}-1}^{})}^{2}\Big).
\]
This completes the proof.
\end{proof}
Finally, we state and prove the main result of this work. It reads as follows.
\begin{theorem} \label{thm:main_sp}
Let $u$ be the solution of \eqref{PDEModel} and $F_{ms,1}(U_{ms})$ be the solution of \eqref{localized_sol} with the numbers of spatial and temporal oversampling layers being $N_s$ and $M$, respectively. 
We have
\begin{eqnarray*}
\|u-F_{ms,1}(U_{ms})\|_{V} & \leq & C_{0}(1+(\frac{\Delta t}{H^{2}})^{\frac{1}{2}})\|\tilde{\kappa}^{-\frac{1}{2}}f\|_{L^{2}}+\|\tilde{\kappa}^{-\frac{1}{2}}\partial_{t}u\|_{L^{2}}\\
 &  & +C\Big((1+\tilde{E}^{-1})^{1-M}+(1+E^{-1})^{1-N}\Big)^{\frac{1}{2}}H^{-\frac{d}{2}}\Big(\|F_{1}(U_{ms})\|_{V}+\|F_{1}(U_{ms})\|_{s}\Big).
\end{eqnarray*}
Moreover, if $C_{\kappa}(1+C_{\kappa}^{2})^{\frac{1}{2}}C^{\frac{1}{2}}\Big((1+\tilde{E}^{-1})^{1-M}+(1+E^{-1})^{1-N}\Big)^{\frac{1}{2}}H^{-\frac{d}{2}}\leq\cfrac{1}{2}$,
we have
\begin{eqnarray*}
\|u-F_{ms,1}(U_{ms})\|_{V} & \leq & C_{0}(1+(\frac{\Delta t}{H^{2}})^{\frac{1}{2}})\|\tilde{\kappa}^{-\frac{1}{2}}f\|_{L^{2}}+\|\tilde{\kappa}^{-\frac{1}{2}}\partial_{t}u\|_{L^{2}}\\
 &  & +C\Big((1+\tilde{E}^{-1})^{1-M}+(1+E^{-1})^{1-N}\Big)^{\frac{1}{2}}H^{-\frac{d}{2}}\Big(\|F_{1}(U_{glo})\|_{V}\Big).
\end{eqnarray*}
\end{theorem}
\begin{proof}
Notice that $u-F_{ms,1}(U_{ms})=u-F_{1}(U_{glo})+F_{1}(U_{glo})-F_{1}(U_{ms})+F_{1}(U_{ms})-F_{ms,1}(U_{ms})$, where $F_{1}(U_{glo})$ is the solution to \eqref{gol_sol}. Using triangle inequality, we obtain
\[
\|u-F_{ms,1}(U_{ms})\|_{V}\leq\|u-F_{1}(U_{glo})\|_{V}+\|F_{1}(U_{glo})-F_{1}(U_{ms})\|_{V}+\|F_{1}(U_{ms})-F_{ms,1}(U_{ms})\|_{V}.
\]
We will estimate the above three terms separately.
By Theorem \ref{thm:glo_err}, we obtain the estimate for the first term:
\[
\|u-F_{1}(U_{glo})\|_{V}\leq C_{0}(1+(\frac{\Delta t}{H^{2}})^{\frac{1}{2}})\|\tilde{\kappa}^{-\frac{1}{2}}f\|_{L^{2}}+\|\tilde{\kappa}^{-\frac{1}{2}}\partial_{t}u\|_{L^{2}}.
\]
To estimate $\|F_{1}(U_{ms})-F_{ms,1}(U_{ms})\|_{V}$, we utilize Lemma \ref{lemma:decay_time} and Lemma \ref{lemma:decay_space} to obtain the following estimate.
\begin{align*}
&\|F_{1}(U_{ms})-F_{ms,1}(U_{ms})\|_{V}^{2}\\
=& \sum_{n,K^{i}}\|F_{1}(U_{ms})-F_{ms,1}(U_{ms})\|_{V(t_{n},t_{n+1};K^{i})}^{2}\\
  \leq & C\sum_{n,K^{i}}\Big((1+\tilde{E}^{-1})^{1-M}+(1+E^{-1})^{1-N_s}\Big)\|F_{1}(U_{ms})\|_{\tilde{V}(t_{n-M},t_{n+1}; K_{N_s}^{i}\backslash K_{N_s-1}^{i})}^{2}\\
  \leq & C\Big((1+\tilde{E}^{-1})^{1-M}+(1+E^{-1})^{1-N_s}\Big)H^{-d}\|F_{1}(U_{ms})\|_{\tilde{V}}^{2},
\end{align*}
where $\|.\|_{\tilde{V}}$ is defined as $\|v\|_{\tilde{V}}^{2}:=\|v\|_{V}^{2}+\|v\|_{s}^{2}$. Finally, we only need to 
estimate $\|F_{1}(U_{glo})-F_{1}(U_{ms})\|_V$. Using Cauchy-Schwartz Inequality and the definition of $C_{\kappa}$, we have
\begin{align*}
\|F_{1}(U_{glo})-F_{1}(U_{ms})\|_{V}^{2} & \leq s(F_{2}(U_{glo})-F_{2}(U_{ms}),F_{1}(U_{glo})-F_{1}(U_{ms}))\\
 & \leq\|F_{2}(U_{glo})-F_{2}(U_{ms})\|_{s}\|F_{1}(U_{glo})-F_{1}(U_{ms})\|_{s}\\
 & \leq C_{\kappa}\|F_{2}(U_{glo})-F_{2}(U_{ms})\|_{s}\|F_{1}(U_{glo})-F_{1}(U_{ms})\|_{V}.
\end{align*}
For any $K^{i}\in\mathcal{T}_H$ and $0\leq t_n<t_{n+1}\leq T$, by Lemma \ref{lemma2} we have
\begin{align*}
 & \|F_{2}(U_{glo})-F_{2}(U_{ms})\|_{s(t_{n},t_{n+1};K^i)}^{2}=\|F_{loc,2}^{(n,i)}(U_{ms})-F_{2}(U_{ms})\|_{s(t_{n},t_{n+1};K^i)}^{2}\\
\leq & C\Big((1+\tilde{E}^{-1})^{1-M}+(1+E^{-1})^{1-N_s}\Big)\Big(\|F_{1}(U_{ms})\|_{\tilde{V}(t_{n-M},t_{n+1}; K_{N_s}^{i})}^{2}-\|F_{1}(U_{ms})\|_{\tilde{V}(t_{n-M+1},t_{n+1}; K_{N_s-1}^{i})}^{2}\Big).
\end{align*}
Finally, we obtain
\[
\|F_{1}(U_{glo})-F_{1}(U_{ms})\|_{V}^{2}\leq C_{\kappa}^{2}C\Big((1+\tilde{E}^{-1})^{1-M}+(1+E^{-1})^{1-N_s}\Big)H^{-d}\|F_{1}(U_{ms})\|_{\tilde{V}}^{2}.
\]
Since $\|v\|_{s}\leq C_{\kappa}\|v\|_{V}$ for any $v\in V$ , we have
\[
\|v\|_{\tilde{V}}\leq(1+C_{\kappa}^{2})^{\frac{1}{2}}\|v\|_{V}.
\]
Therefore, we have
\begin{align*}
\|F_{1}(U_{ms})-F_{1}(U_{glo})\|_{V} & \leq C_{\kappa}(1+C_{\kappa}^{2})^{\frac{1}{2}}C^{\frac{1}{2}}\Big((1+\tilde{E}^{-1})^{1-M}+(1+E^{-1})^{1-N_s}\Big)^{\frac{1}{2}}H^{-\frac{d}{2}}\|F_{1}(U_{ms})\|_{V}\\
 & \leq C_{\kappa}(1+C_{\kappa}^{2})^{\frac{1}{2}}C^{\frac{1}{2}}\Big((1+\tilde{E}^{-1})^{1-M}+(1+E^{-1})^{1-N_s}\Big)^{\frac{1}{2}}H^{-\frac{d}{2}}\Big( \\
 &~~~~\|F_{1}(U_{ms})-F_{1}(U_{glo})\|_{V}+\|F_{1}(U_{glo})\|_{V}\Big).
\end{align*}
If $C_{\kappa}(1+C_{\kappa}^{2})^{\frac{1}{2}}C^{\frac{1}{2}}\Big((1+\tilde{E}^{-1})^{1-M}+(1+E^{-1})^{1-N}\Big)^{\frac{1}{2}}H^{-\frac{d}{2}}\leq\cfrac{1}{2}$,
we have
\[
\|F_{1}(U_{ms})-F_{1}(U_{glo})\|_{V}\leq2C_{\kappa}(1+C_{\kappa}^{2})^{\frac{1}{2}}C^{\frac{1}{2}}\Big((1+\tilde{E}^{-1})^{1-M}+(1+E^{-1})^{1-N}\Big)^{\frac{1}{2}}H^{-\frac{d}{2}}\|F_{1}(U_{glo})\|_{V}.
\]

\end{proof}
\textbf{Remark:} If the multiscale partition of unity $\chi_{i}$ is replaced by the bilinear partition of unity in the definition of $\tilde{\kappa}$, one can easily prove that with an appropriate choice of the spatial and temporal oversampling layers, we have 
\begin{equation*}
\|u-F_{ms,1}(U_{ms})\|_{V}  \lesssim  C H\|\kappa f\|_{L^{2}}+CH\|\kappa \partial_{t}u\|_{L^{2}}
  +CH\|F_{1}(U_{glo})\|_{V}.
\end{equation*}

\section{Numerical Results}\label{sec:numerical_SP}
In this section, we present numerical results for the proposed numerical method. We shall solve the system \eqref{PDEModel} in the unit square $\Omega=[0,1]^2$ with total time $T=1.0$. The source term $f(t,x)$ is chosen to be a smooth function $f(t,x):=x_1x_2t$. The permeability filed $\kappa(t,x)$ is  time-dependent. We will test our numerical methods with two kinds of permeability fields: slow moving permeability in Experiment 1 and a faster moving permeability in Experiment 2.

Let $\mathcal{T}_H\times\mathcal{T}_{\Delta t} $ be a decomposition of the space-time domain $ \Omega\times [0,T]$ into non-overlapping shape-regular cubic elements with maximal spatial mesh size $H$ and temporal mesh size $\Delta t$. These coarse cubic elements are further partitioned into a collection of connected fine cubic elements $\mathcal{T}_h\times\mathcal{T}_{\delta t}$ using fine spatial mesh size $h$ and temporal mesh size $\delta t$. We define $V_{h,\delta}$ to be a conforming piecewise affine finite element associated with $\mathcal{T}_{h}\times\mathcal{T}_{\delta t}$. Since there is no analytic solution to system \eqref{PDEModel}, we are going to find an approximation of its exact solutions. To this end, we use the constructed fine mesh and conforming space-time finite element method to obtain the reference solutions $U_{h,\delta}$. The multiscale solutions $U^{\ell_x,\ell_t}_{\text{ms}}$ are obtained using our proposed space-time NLMC method with spatial oversampling layers number being $\ell_x$ and temporal oversampling layers number being $\ell_t$.  We use $\tilde{U}_{h,\delta,t}$ to denote the snapshot of the reference solutions at time $t$ and $\tilde{U}^{\ell_x,\ell_t}_{\text{ms},t}$ to denote the snapshot of  multiscale solutions using spatial oversampling layer $\ell_x$ and temporal oversampling layer $\ell_t$ at time $t$. To simply notations, we use  $\tilde{U}^{\ell}_{\text{ms},t}$ to denote $\tilde{U}^{\ell_x,\ell_t}_{\text{ms},t}$ when the number of spatial oversampling layers equals that of temporal oversampling layers, that is,  $\ell:=\ell_x=\ell_t$.

We introduce the following notations to calculate the errors. The relative errors for the multiscale solution in $L^2$-norm and $H^1_{\kappa}$-norm are
 \[
 \text{Rel}_{L^2}^\ell :=\frac{ \norm{U_{h,\delta}-U^{\ell_x,\ell_t}_{\text{ms}}}_{L^2}}{ \norm{U_{h,\delta}} _{L^2}} \times 100 \quad\text{ and }\quad
  \text{Rel}_{H_{\kappa}^1}^\ell :=\frac{ \norm{U_{h,\delta}-U^{\ell_x,\ell_t}_{\text{ms}}}_{H_{\kappa}^1}}{ \norm{U_{h,\delta}} _{H_{\kappa}^1 }}\times 100.
  \]

\subsection{Experiment 1: Slow moving  permeability}
In this experiment, we choose the permeability with 1 channel moving slowly in horizontal direction. Let
\[
S_1:=\{(x_1,x_2,t): 0.375<x_1<0.6094,0.50<x_2<0.5156, 0\leq t< 0.5 \}
\]
and 
\[
S_2:=\{(x_1,x_2,t): 0.3906<x_1<0.6250, 0.50<x_2<0.5156, 0.5\leq t\leq 1.0 \}.
\]
The permeability $\kappa(x_1,x_2,t)$ is defined as
\begin{equation*}
\kappa(x_1,x_2,t):=
\left\{
\begin{aligned}
1000,&\text{ if } (x_1,x_2,t)\in S_1\cup S_2, \\
1,& \text{ otherwise. } 
\end{aligned}
\right .
\end{equation*}

We present the permeability field at time $t=0$ and $t=0.5$ in Figure \ref{Exp1: permeability} for an illustration.
\begin{figure}[H]
		\centering
		\includegraphics[trim={2.9cm 1.8cm 1.0cm 1.5cm},clip,width=0.24 \textwidth]{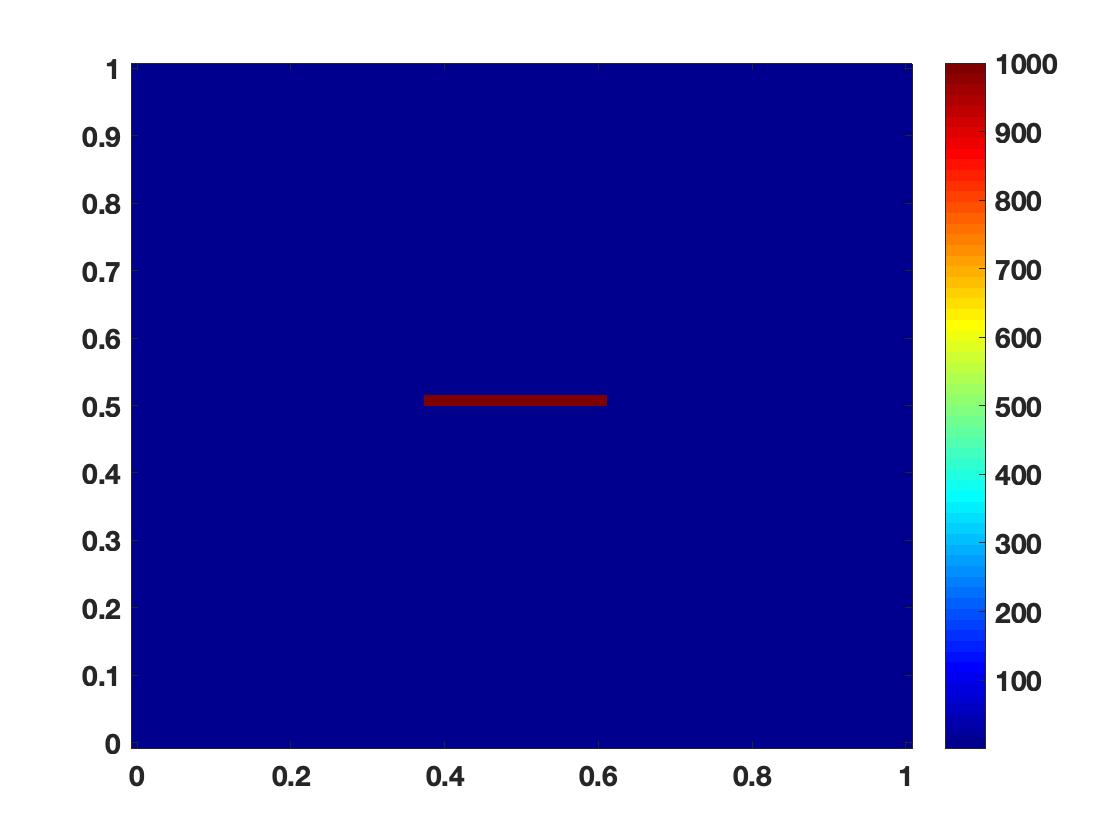}
		\includegraphics[trim={2.9cm 1.8cm 1.0cm 1.5cm},clip,width=0.24 \textwidth]{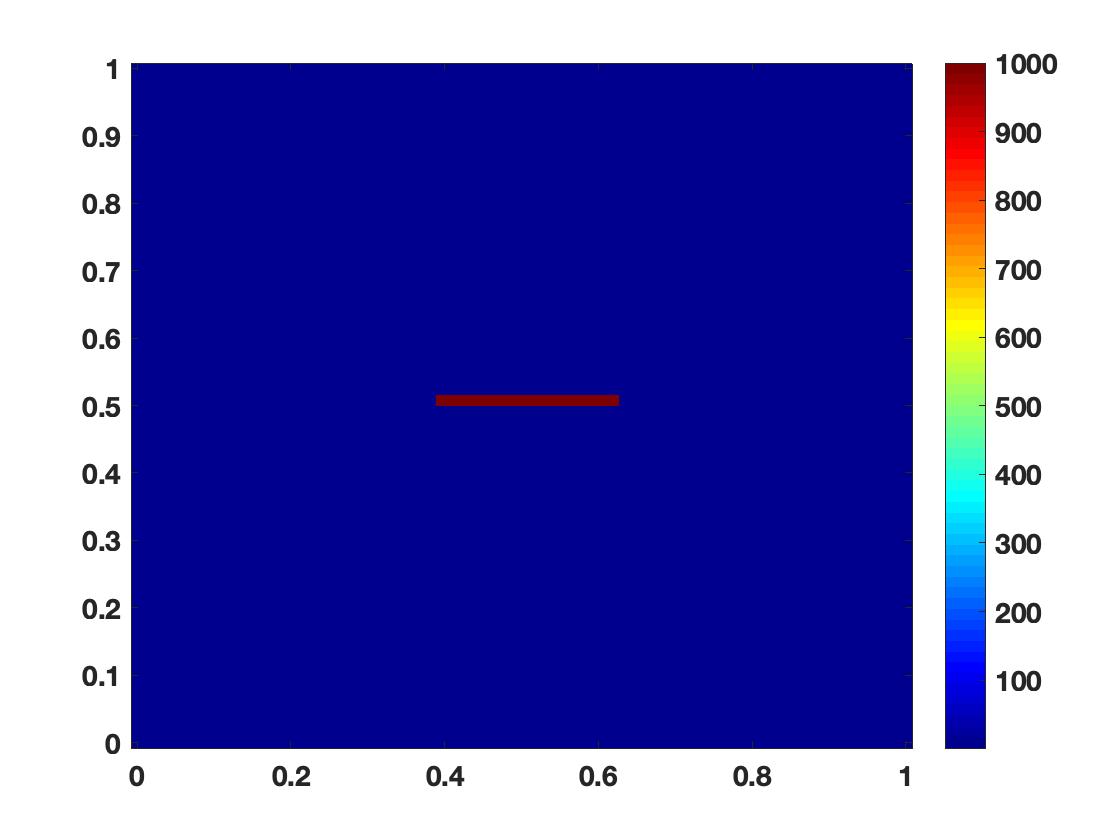}
		\caption{Permeability field $\kappa(x_1,x_2,0)$ and $\kappa(x_1,x_2,0.5)$ for Experiment 1.}
		\label{Exp1: permeability}
\end{figure}
We choose fine spatial mesh size $h=2^{-6}$, fine temporal mesh size $\delta t=0.01$, coarse spatial mesh size $H=2^{-3}$ and coarse temporal mesh size $\Delta t=0.1$. The number of spatial and temporal oversampling layers $\ell_x$ and $\ell_t$ are chosen to be $\ell=\ell_x=\ell_t \in \{1,2,\cdots,5\}$.
The snapshot of reference solutions $\tilde{U}_{h,\delta,t}$ and multiscale solutions $\tilde{U}^\ell_{\text{ms},t}$ for  $t= 0.25,0.5,0.75,1.0$ and $\ell=1,2,3$ are plotted in Figure \ref{Exp1: fine solution} and Figure \ref{Exp1: multiscale_sol}, respectively. 
\begin{figure}[H]
		\centering
		\includegraphics[trim={2.9cm 1.8cm 1.0cm 0.6cm},clip,width=0.24 \textwidth]{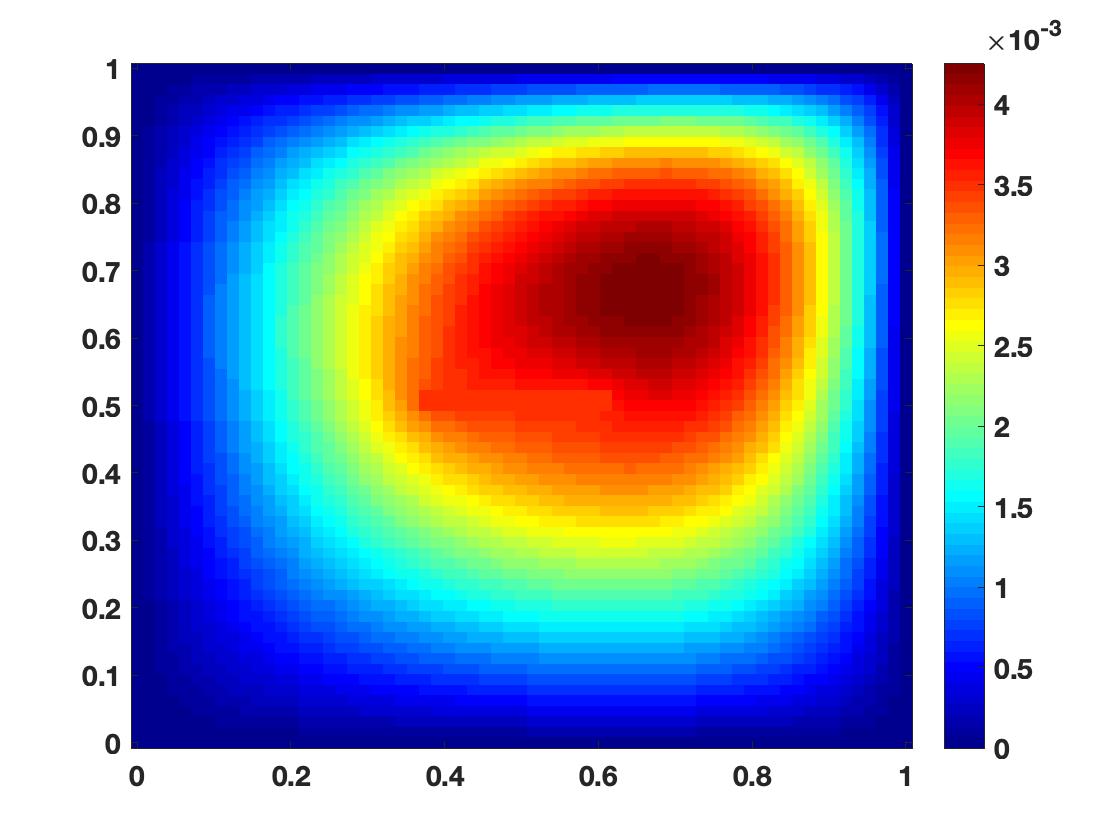}
		\includegraphics[trim={2.9cm 1.8cm 1.0cm 0.6cm},clip,width=0.24 \textwidth]{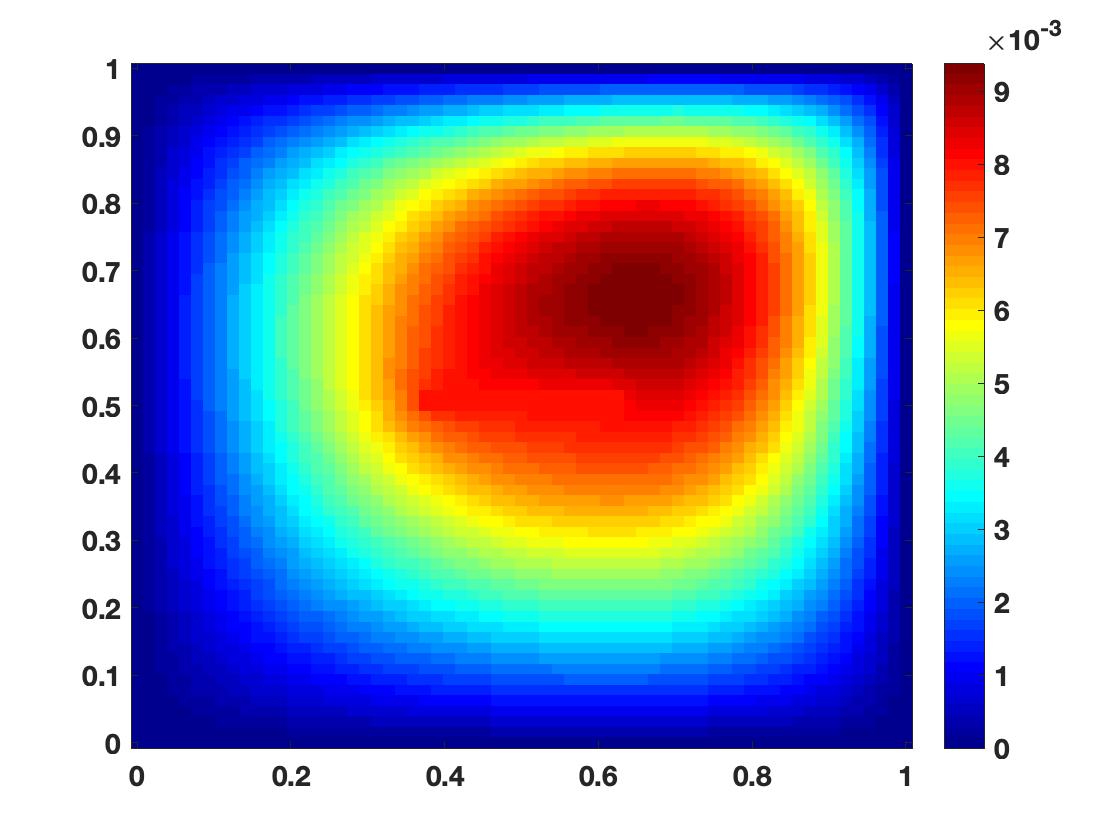}
		\includegraphics[trim={2.9cm 1.8cm 1.0cm 0.6cm},clip,width=0.24 \textwidth]{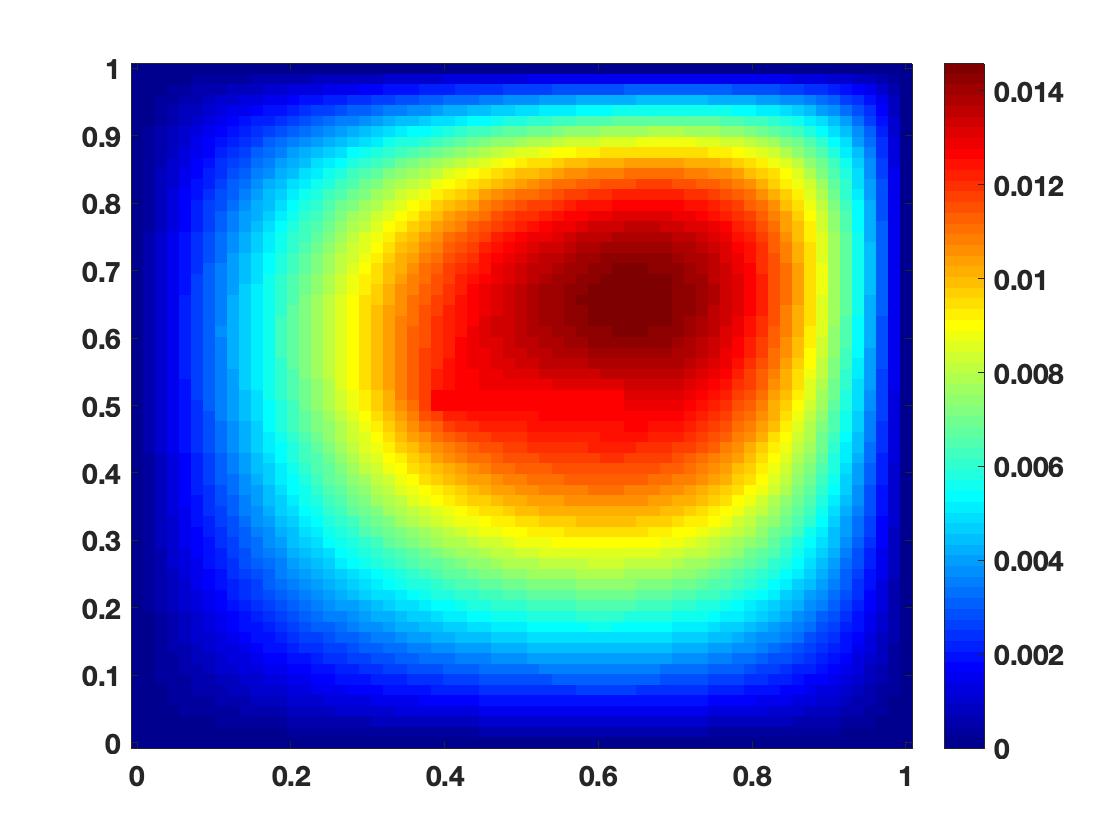}
		\includegraphics[trim={2.9cm 1.8cm 1.0cm 0.6cm},clip,width=0.24 \textwidth]{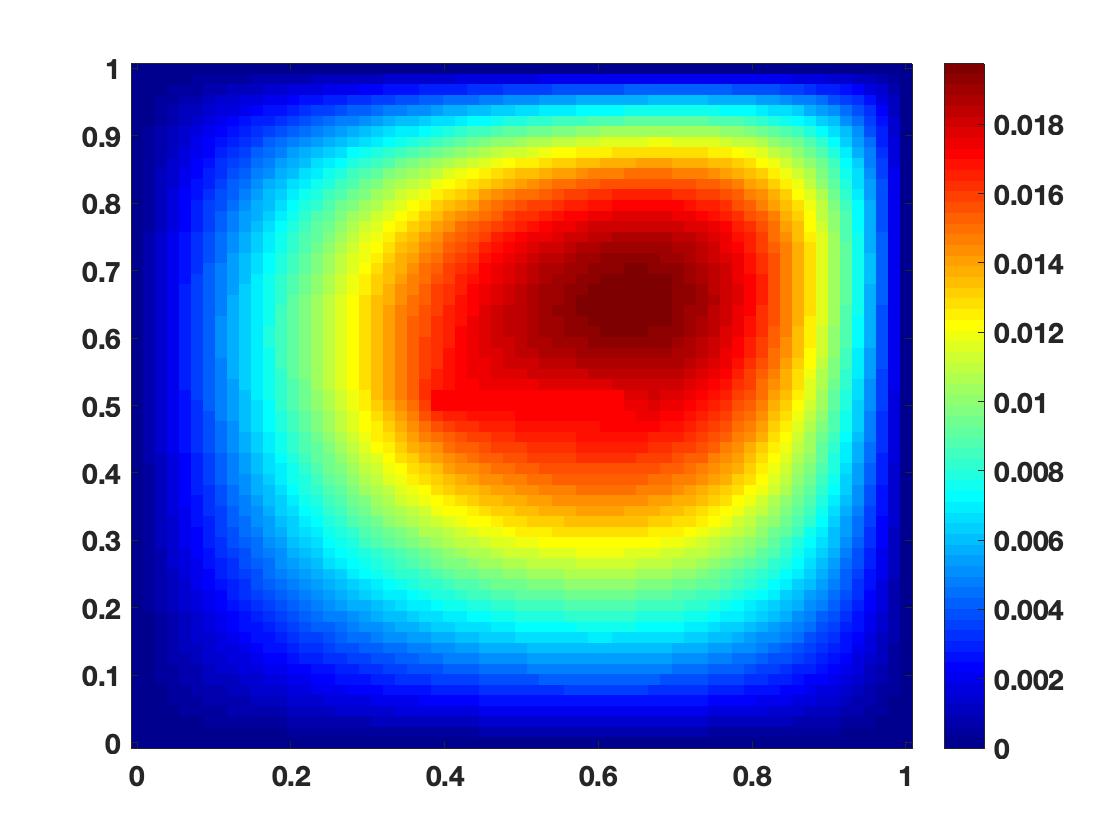}
		\caption{Snapshot of the reference solutions  $\tilde{U}_{h,\delta,t}$ for  $t= 0.25,0.5,0.75,1.0$.}
		\label{Exp1: fine solution}
\end{figure}

\begin{figure}[H]
		\centering
		\includegraphics[trim={2.9cm 1.8cm 1.0cm 0.6cm},clip,width=0.24 \textwidth]{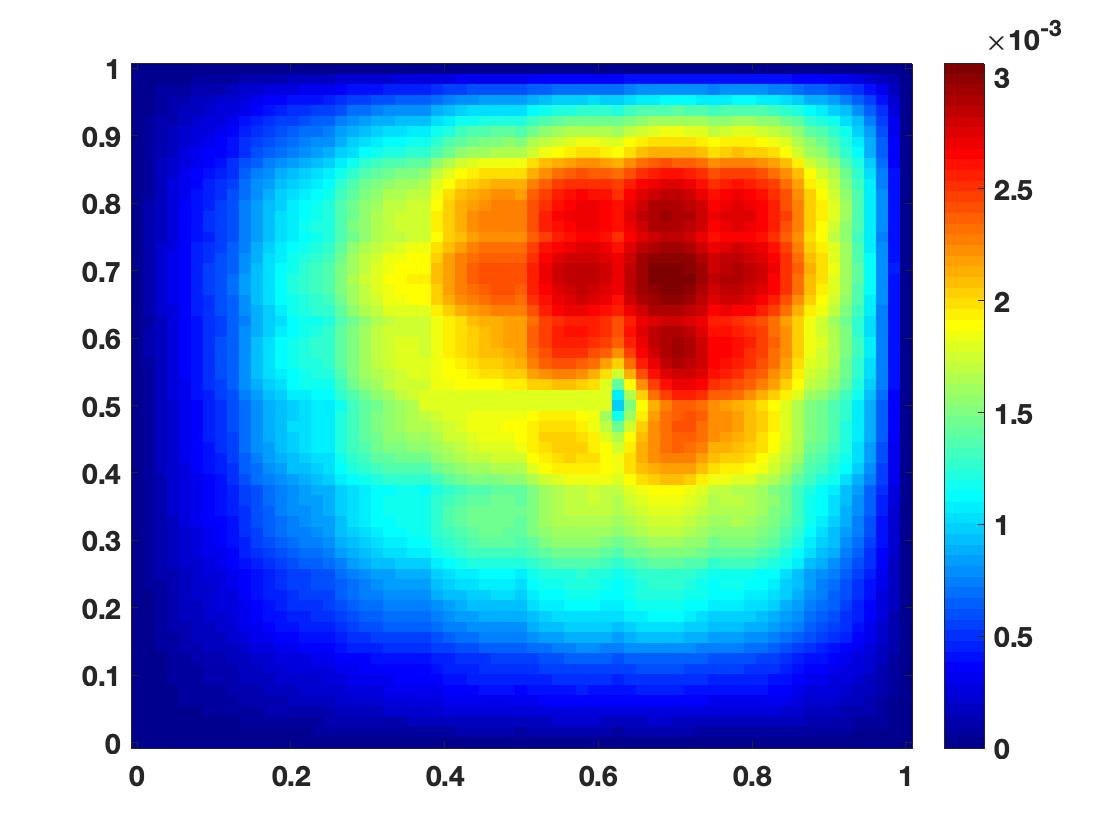}
		\includegraphics[trim={2.9cm 1.8cm 1.0cm 0.6cm},clip,width=0.24 \textwidth]{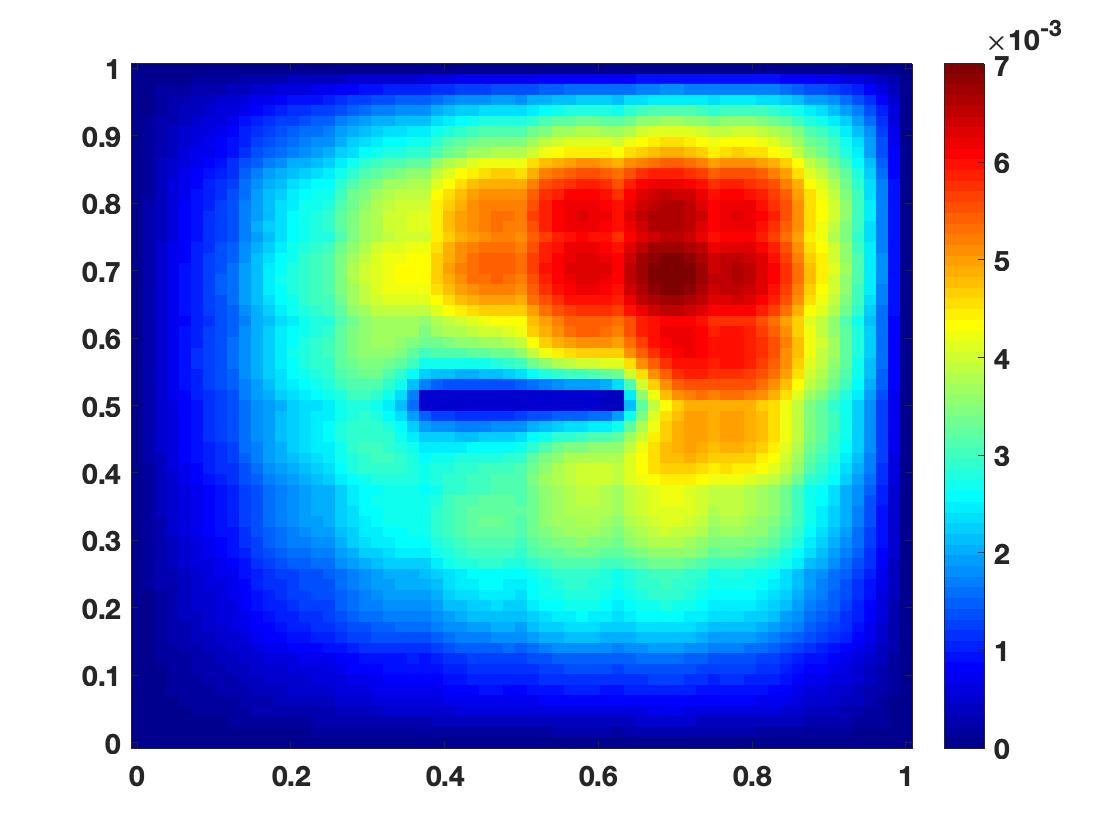}
		\includegraphics[trim={2.9cm 1.8cm 1.0cm 0.6cm},clip,width=0.24 \textwidth]{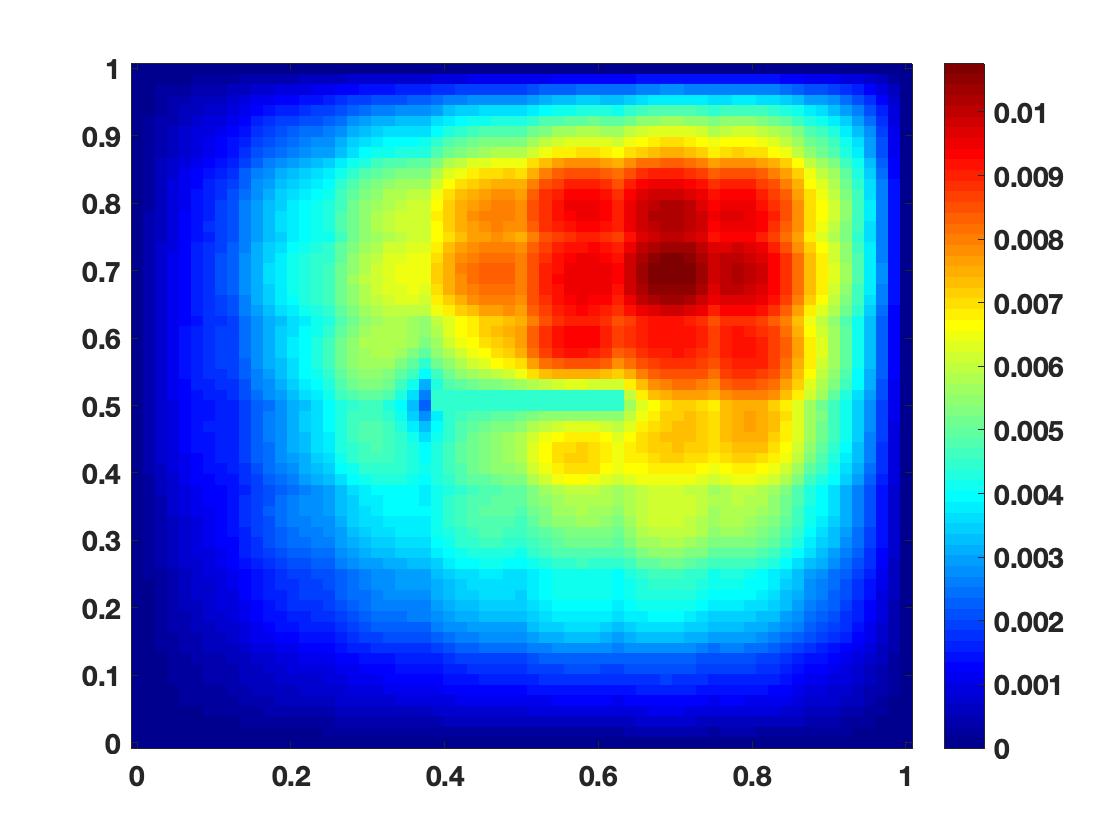}
		\includegraphics[trim={2.9cm 1.8cm 1.0cm 0.6cm},clip,width=0.24 \textwidth]{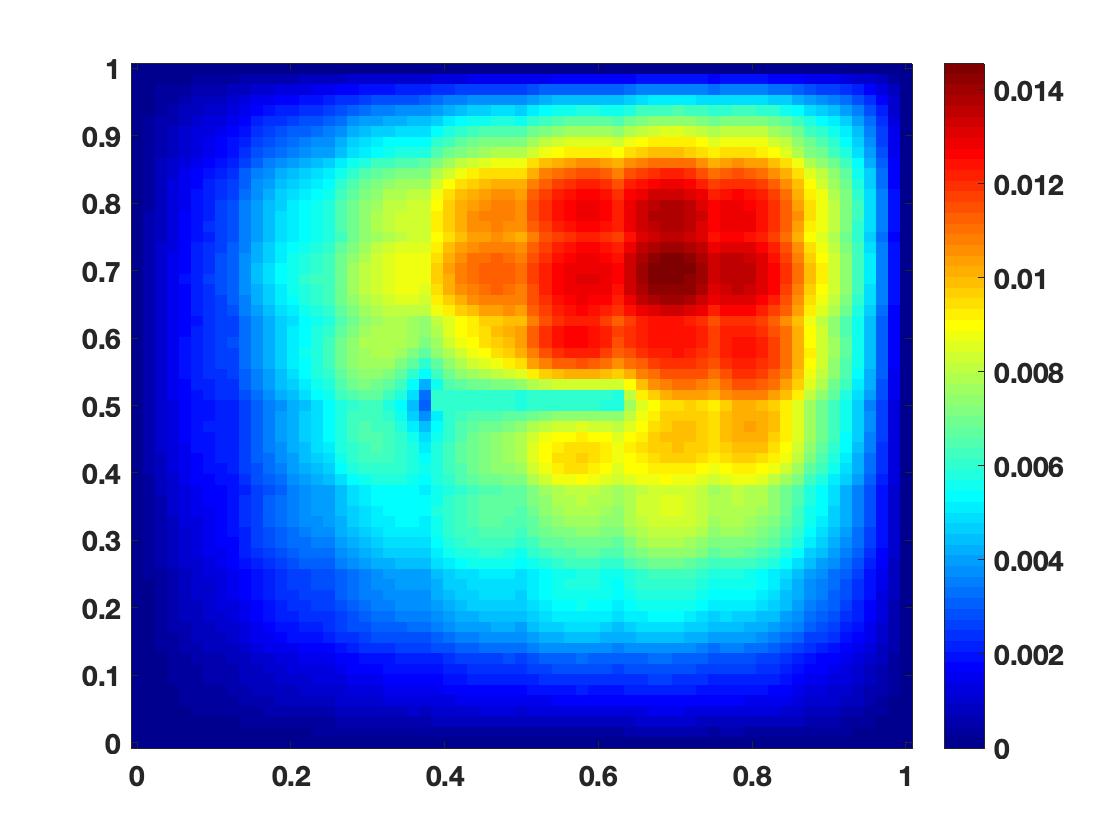}\\
		\includegraphics[trim={2.9cm 1.8cm 1.0cm 0.6cm},clip,width=0.24 \textwidth]{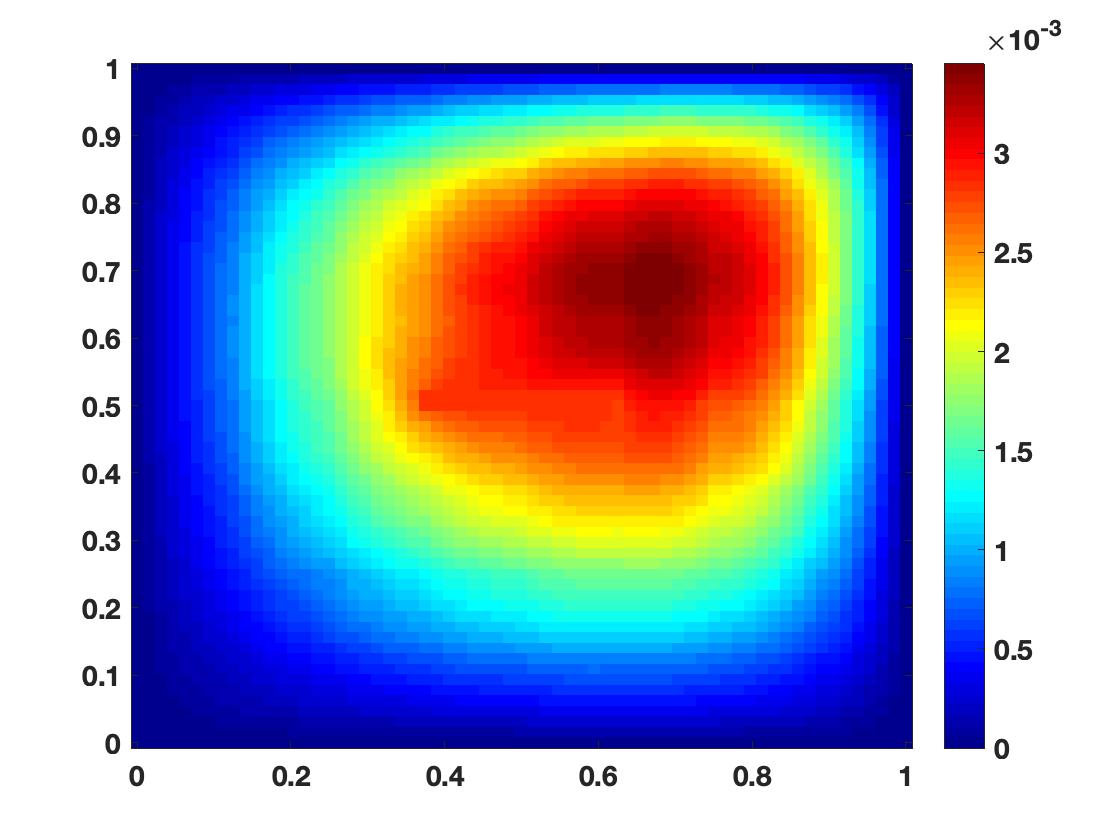}
		\includegraphics[trim={2.9cm 1.8cm 1.0cm 0.6cm},clip,width=0.24 \textwidth]{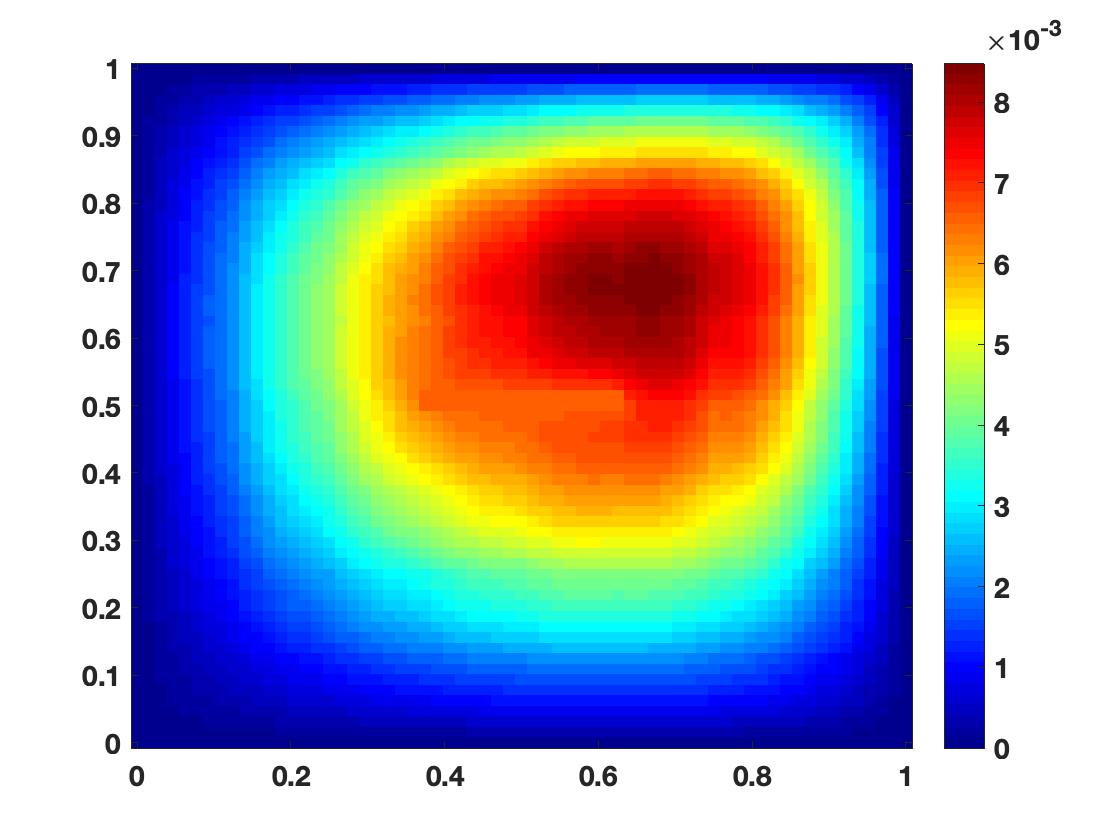}
		\includegraphics[trim={2.9cm 1.8cm 1.0cm 0.6cm},clip,width=0.24 \textwidth]{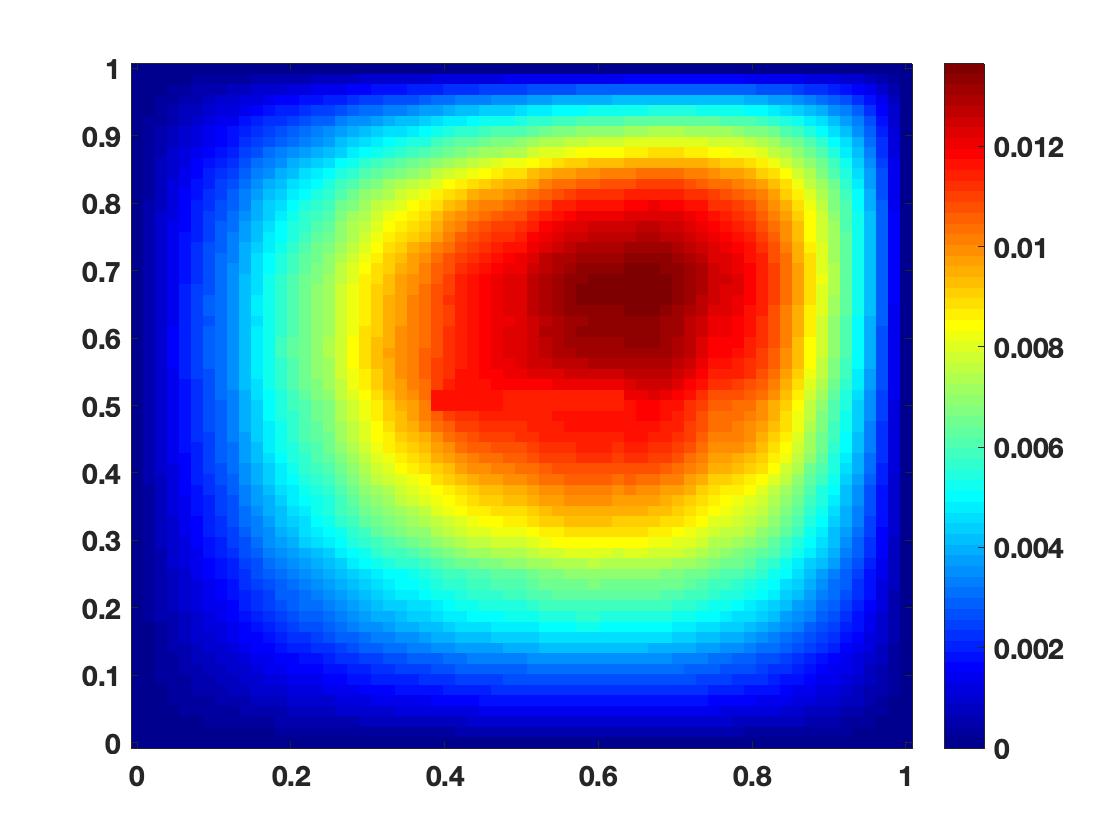}
		\includegraphics[trim={2.9cm 1.8cm 1.0cm 0.6cm},clip,width=0.24 \textwidth]{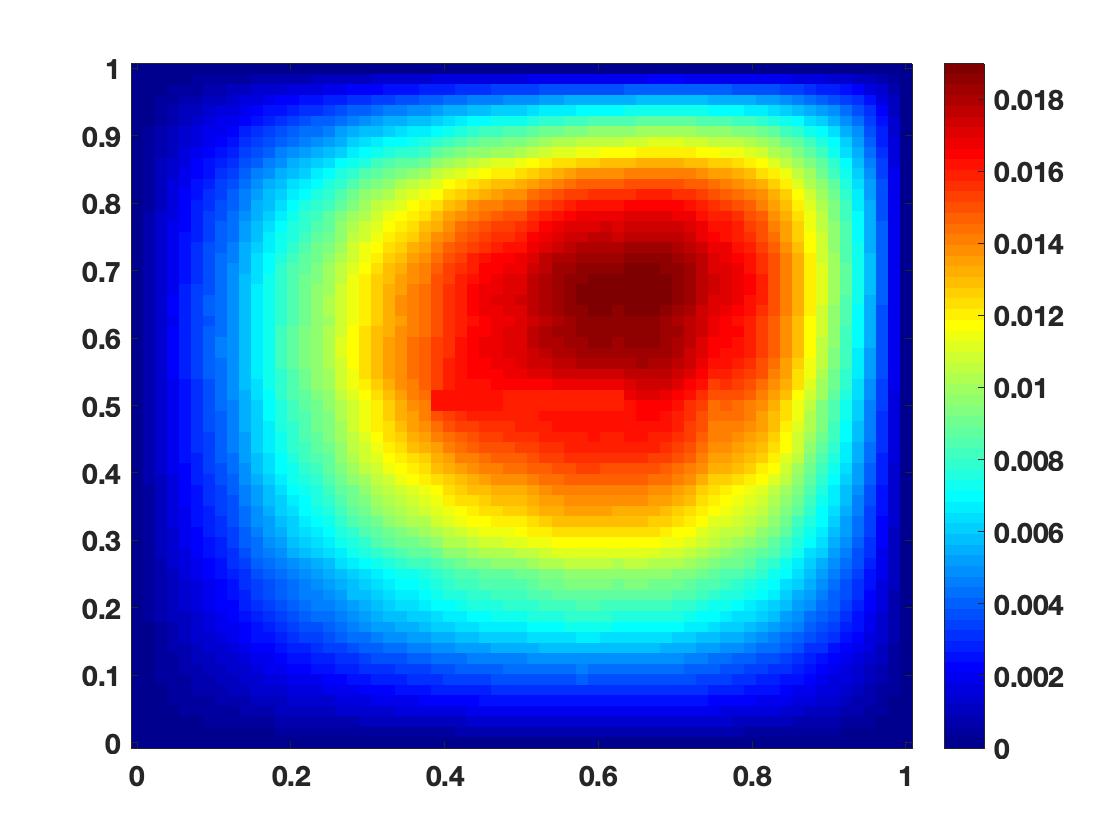}\\
		\includegraphics[trim={2.9cm 1.8cm 1.0cm 0.6cm},clip,width=0.24 \textwidth]{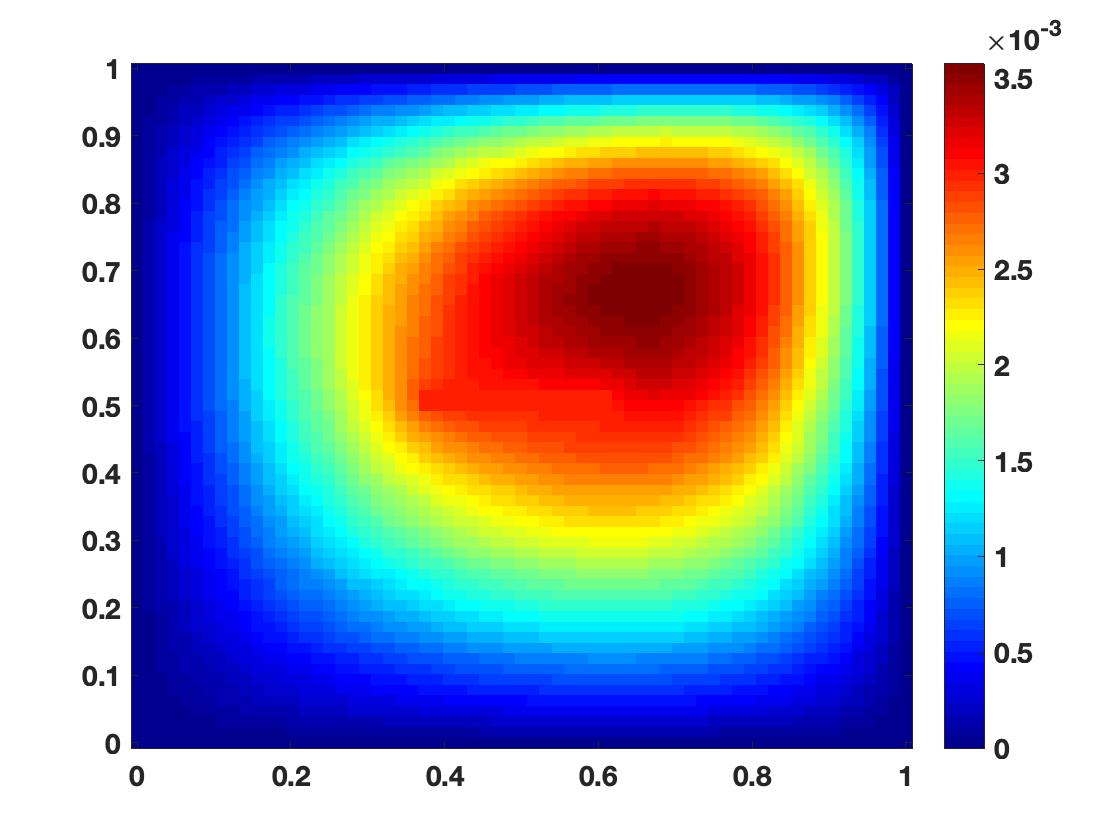}
		\includegraphics[trim={2.9cm 1.8cm 1.0cm 0.6cm},clip,width=0.24 \textwidth]{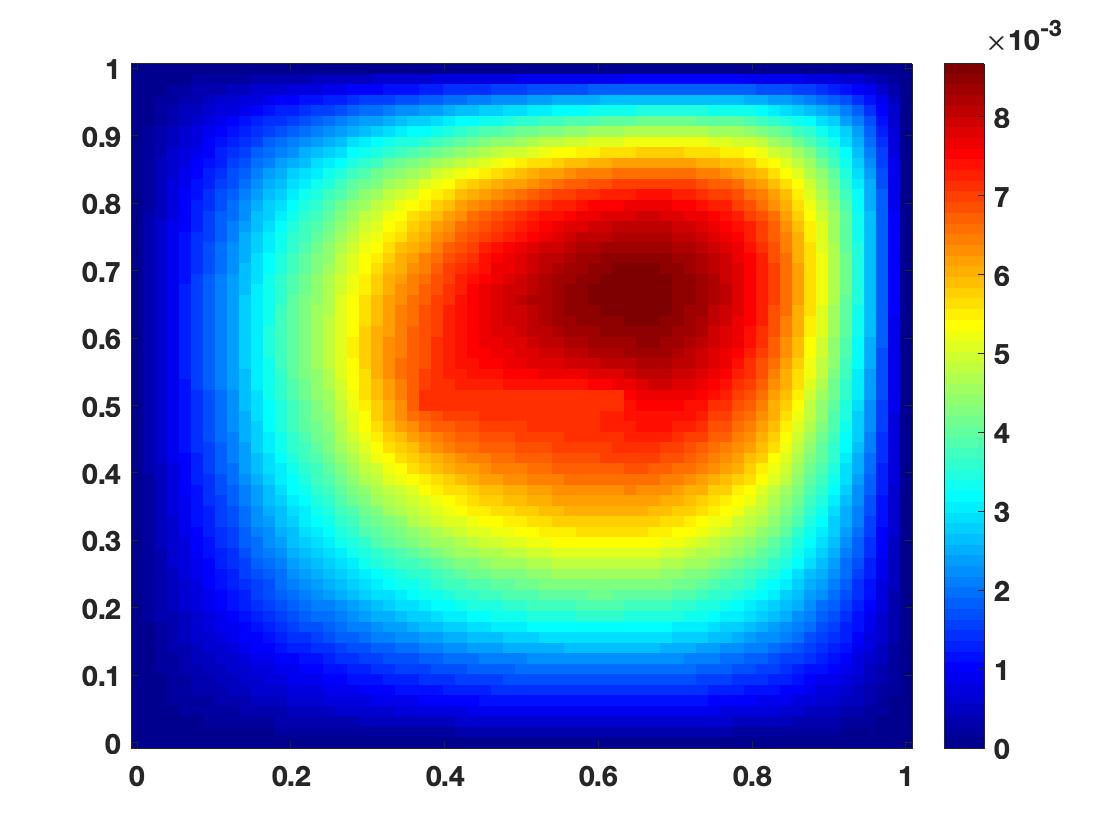}
		\includegraphics[trim={2.9cm 1.8cm 1.0cm 0.6cm},clip,width=0.24 \textwidth]{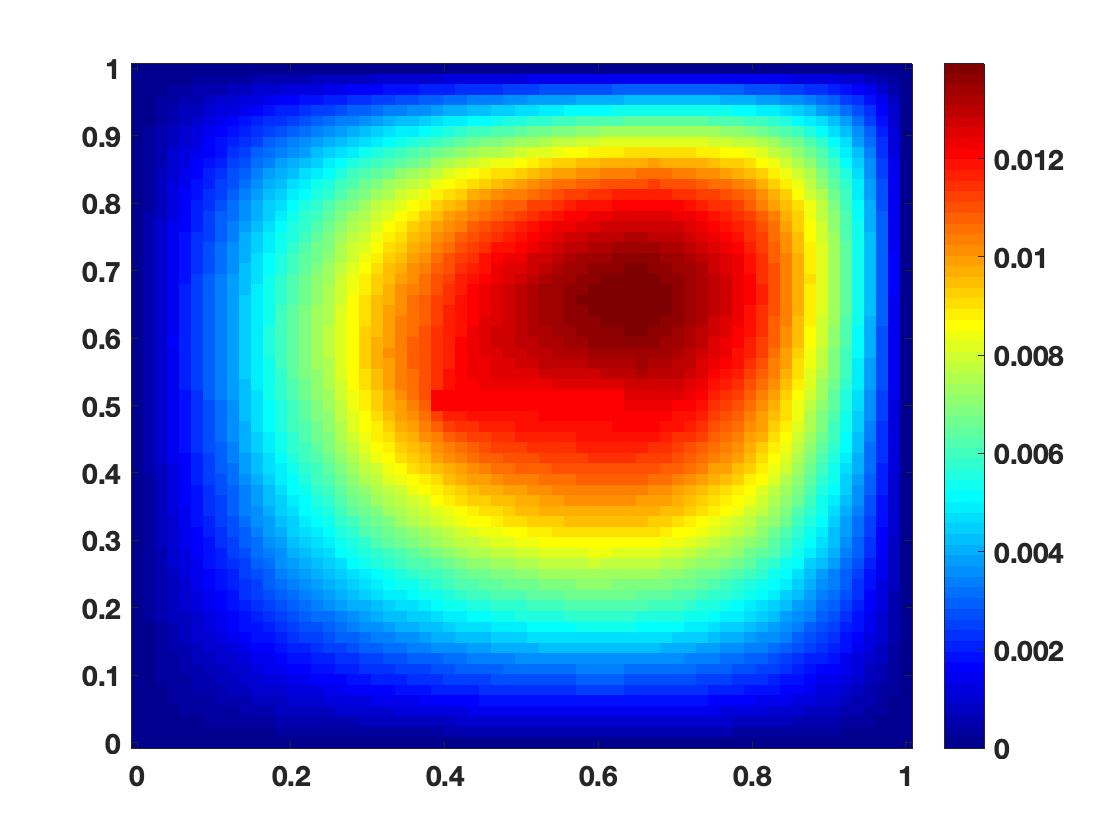}
		\includegraphics[trim={2.9cm 1.8cm 1.0cm 0.6cm},clip,width=0.24 \textwidth]{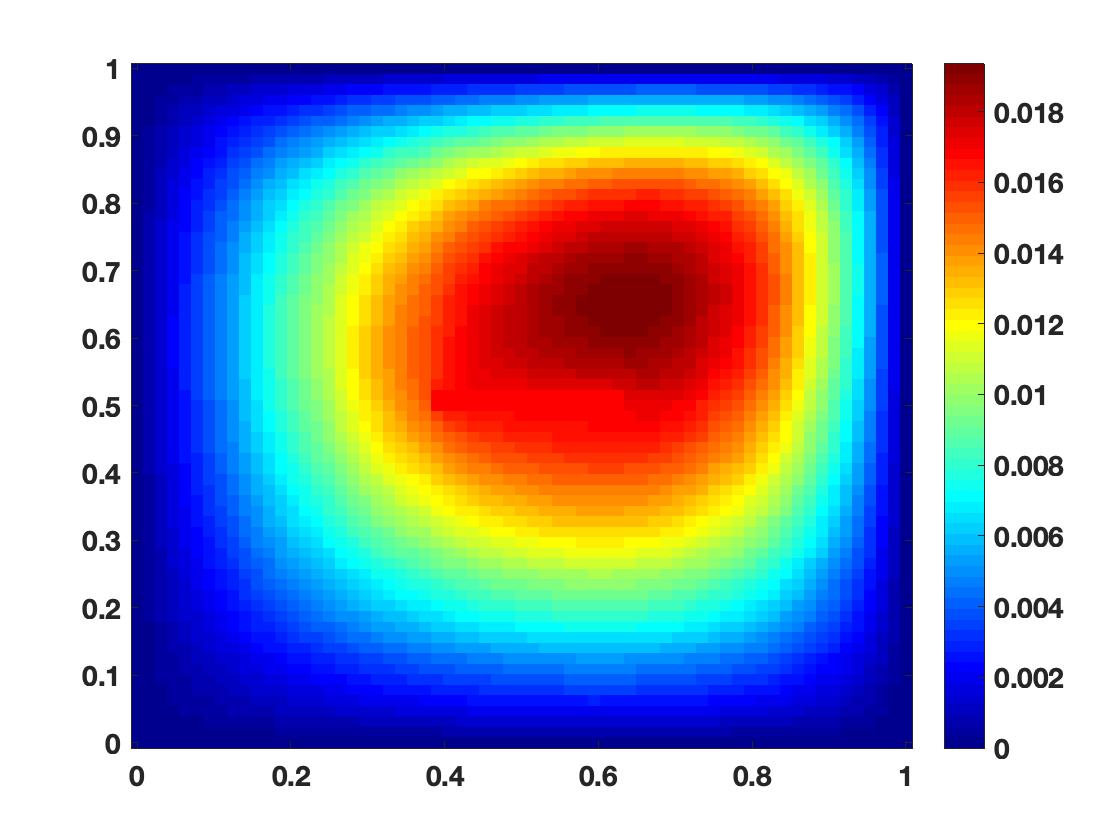}
		\caption{Snapshot of the multiscale solutions $\tilde{U}^\ell_{\text{ms},t}$ at $t=0.25,0.5,0.75,1.0$ with oversampling layer $\ell=1$ (top), $\ell=2$ (middle), $\ell=3$ (bottom).}
		\label{Exp1: multiscale_sol}
\end{figure}

The convergence history in relative $L^2$-norm and relative $H_{\kappa}^1$-norm with oversampling layers number $\ell=1,2,\cdots,5$ are presented in Table \ref{Table: Exp1_RelativeError}.
\begin{table}[H]
	\begin{center}
	\begin{tabular}{|c|c|c|c|c|c|c|}
	\hline
	$\ell$ & $\text{Rel}_{H_{\kappa}^1}^\ell $&$\text{Rel}_{L^2}^\ell$
	\\ \hline
	 1 & 53.6304&35.6654 \\
 2 &  15.2632 &5.0203\\
 3 &  7.2096 & 3.3863\\
 4 & 4.3655  &2.7838\\
 5 &  3.4061 &2.5349 \\
	\hline
	\end{tabular}
	\end{center}
	\vspace{-.4cm}
	\caption{Convergence history of  Experiment 1.}
	\label{Table: Exp1_RelativeError}
	\end{table}
Theorem \ref{thm:main_sp} shows that the approximated numerical solution can be improved by increasing the oversampling size. This is confirmed by this numerical experiment as shown in Table \ref{Table: Exp1_RelativeError}.
\subsection{Experiment 2: Faster moving  permeability }
In this experiment, we choose a permeability with faster moving channels. To define the permeability for this experiment, we introduce 4 sets $S_1, S_2, S_3$ and $S_4$ as follows.
\begin{align*}
S_{1}&:= \cup_{k=1}^{25} \{(x_1,x_2,t): 0.09+0.01k\leq x_1\leq 0.11+0.01k, 0.30\leq x_2\leq 0.70, 0.04(k-1)\leq t\leq 0.04k\}, \\
S_{2}&:=\cup_{k=1}^{20} \{(x_1,x_2,t): 0.39+0.01k\leq x_1\leq 0.79+0.01k, 0.15\leq x_2\leq 0.17, 0.05(k-1)\leq t\leq 0.05k\}, \\
S_{3}&:=\cup_{k=1}^{25} \{(x_1,x_2,t): 0.29+0.01k\leq x_1\leq 0.44+0.01k, 0.19+0.01k\leq x_2\leq 0.21+0.01k, \\
&~~~~~~~~~~~~~~~~~~~~~~~~~~~0.04(k-1)\leq t\leq 0.04k\}, \\
S_{4}&:=\cup_{k=1}^{10} \{(x_1,x_2,t): 0.59+0.01k\leq x_1\leq 0.94+0.01k, 0.63+0.01k\leq x_2\leq 0.65+0.01k,\\
&~~~~~~~~~~~~~~~~~~~~~~~~~~~0.1(k-1)\leq t\leq 0.1k\}.
\end{align*}
The permeability $\kappa(x_1,x_2,t)$ is defined as below:
\begin{equation*}
\kappa(x_1,x_2,t):=
\left\{
\begin{aligned}
1000,&\text{ if } (x_1,x_2,t)\in S_1\cup S_2\cup S_3\cup S_4, \\
1,& \text{ otherwise. } 
\end{aligned}
\right .
\end{equation*}

We present the permeability field at time $t=0, 0.5, 0.8$ and $t=1.0$ in Figure \ref{Exp2: permeability} for an illustration.
\begin{figure}[H]
		\centering
		\includegraphics[trim={2.9cm 1.8cm 1.0cm 2cm},clip,width=0.24 \textwidth]{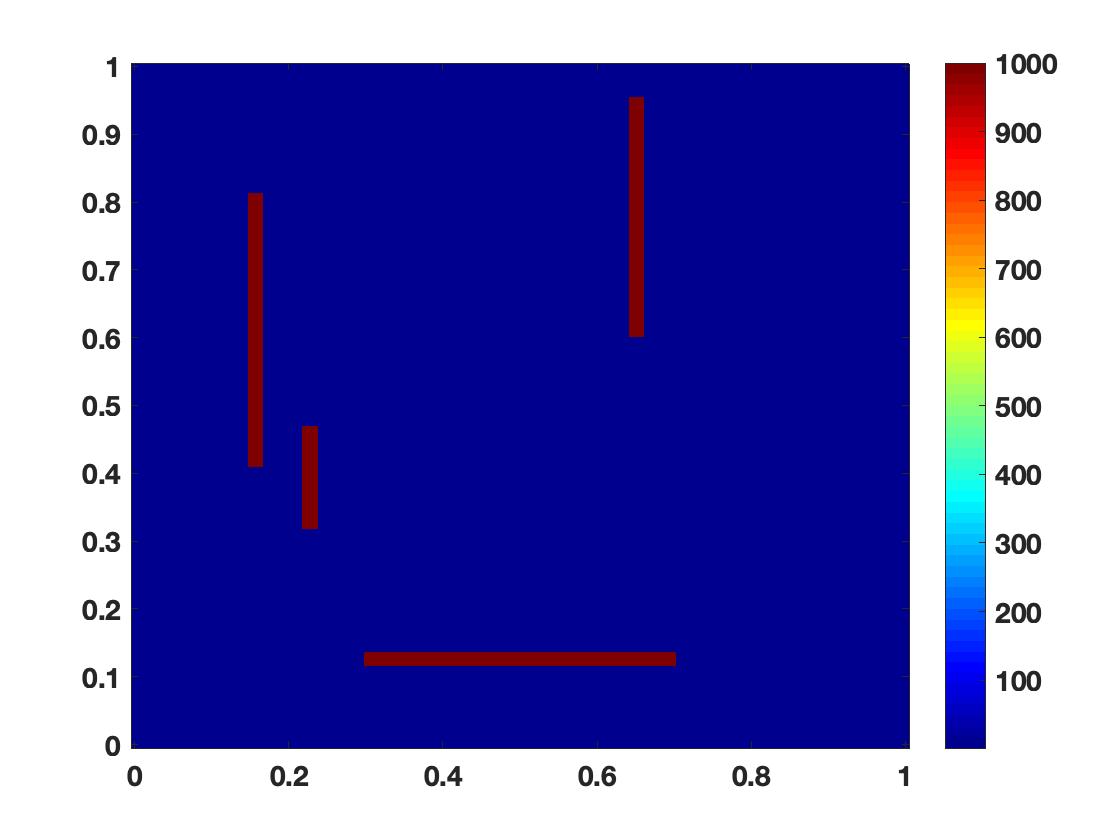}
		\includegraphics[trim={2.9cm 1.8cm 1.0cm 2cm},clip,width=0.24 \textwidth]{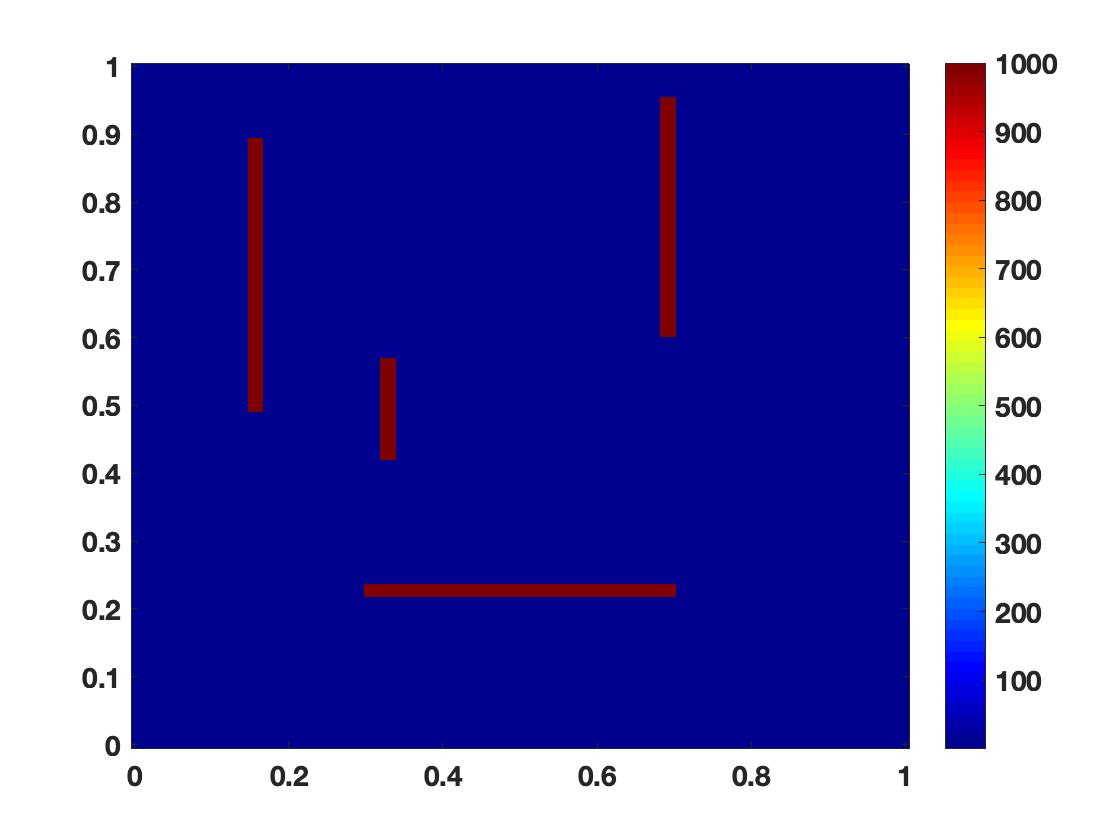}
		\includegraphics[trim={2.9cm 1.8cm 1.0cm 2cm},clip,width=0.24 \textwidth]{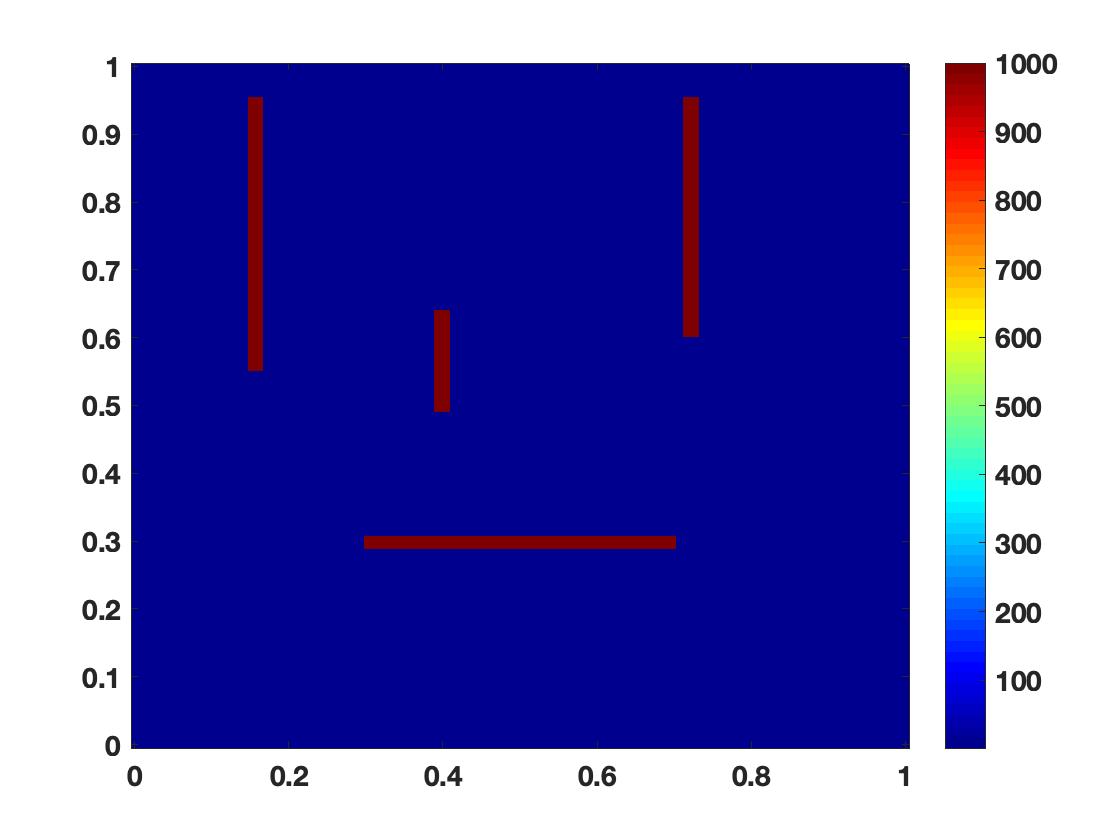}
		\includegraphics[trim={2.9cm 1.8cm 1.0cm 2cm},clip,width=0.24 \textwidth]{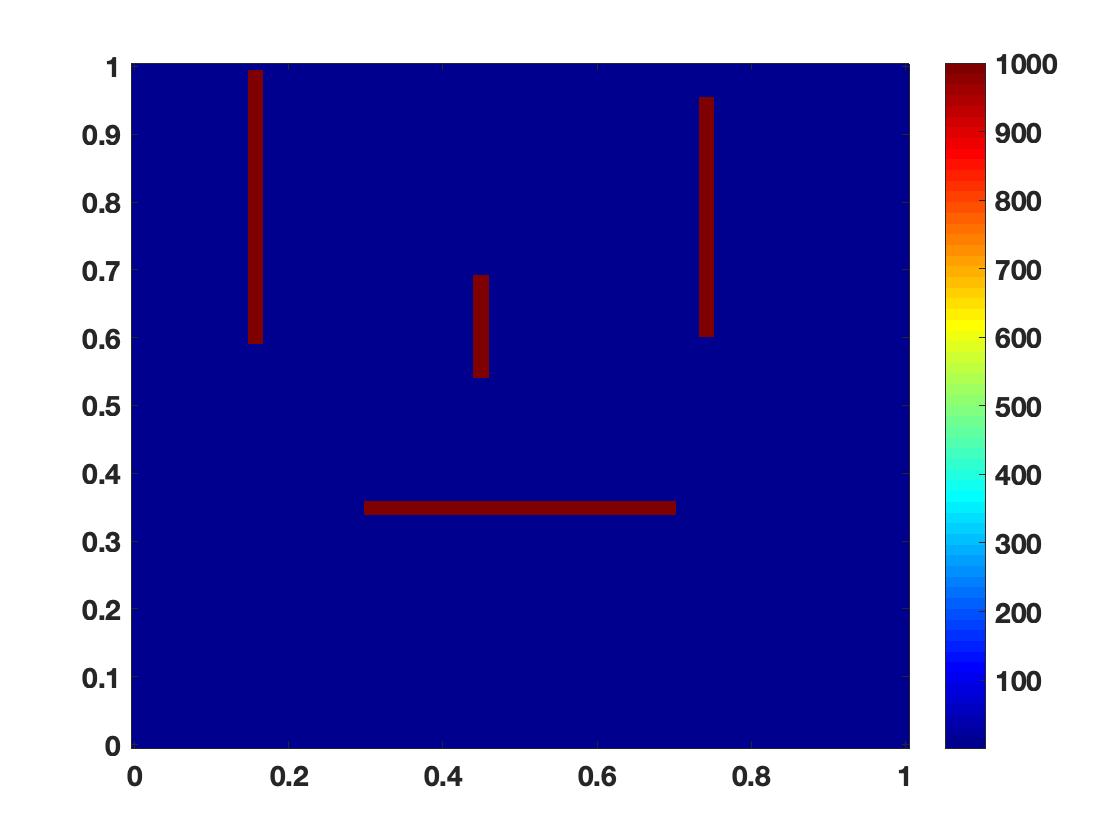}
		\caption{Permeability field $\kappa(x_1,x_2,0), \kappa(x_1,x_2,0.5), \kappa(x_1,x_2,0.8)$ and $\kappa(x_1,x_2,1.0).$
		\label{Exp2: permeability}}
\end{figure}

The spatial and temporal fine mesh size we use to approximate the exact solution is $h=0.01$ and $\delta t=0.01$.
The snapshot of reference solutions $\tilde{U}_h(t)$ to approximate exact solution $u(x,t)$ at time $t= 0.2,0.5,0.8,1.0$ are plotted as below.
\begin{figure}[H]
		\centering
		\includegraphics[trim={2.9cm 1.6cm 1.0cm 0.6cm},clip,width=0.24 \textwidth]{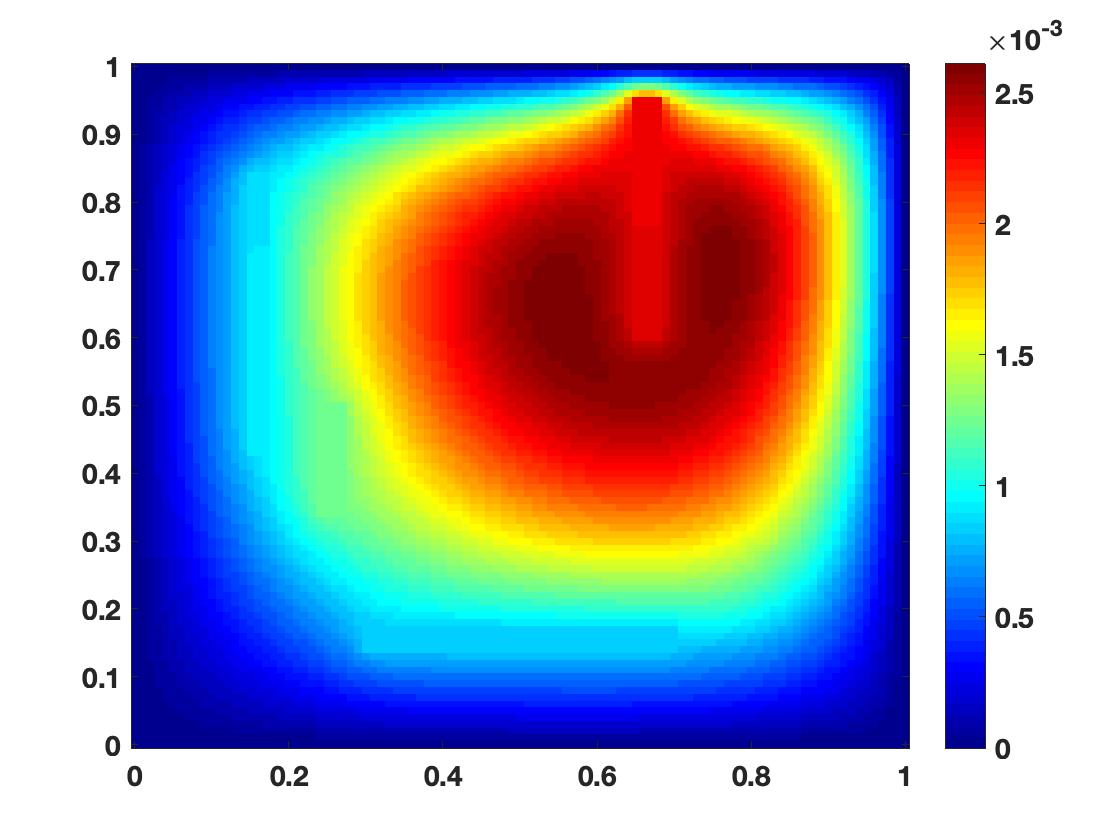}
		\includegraphics[trim={2.9cm 1.6cm 1.0cm 0.6cm},clip,width=0.24 \textwidth]{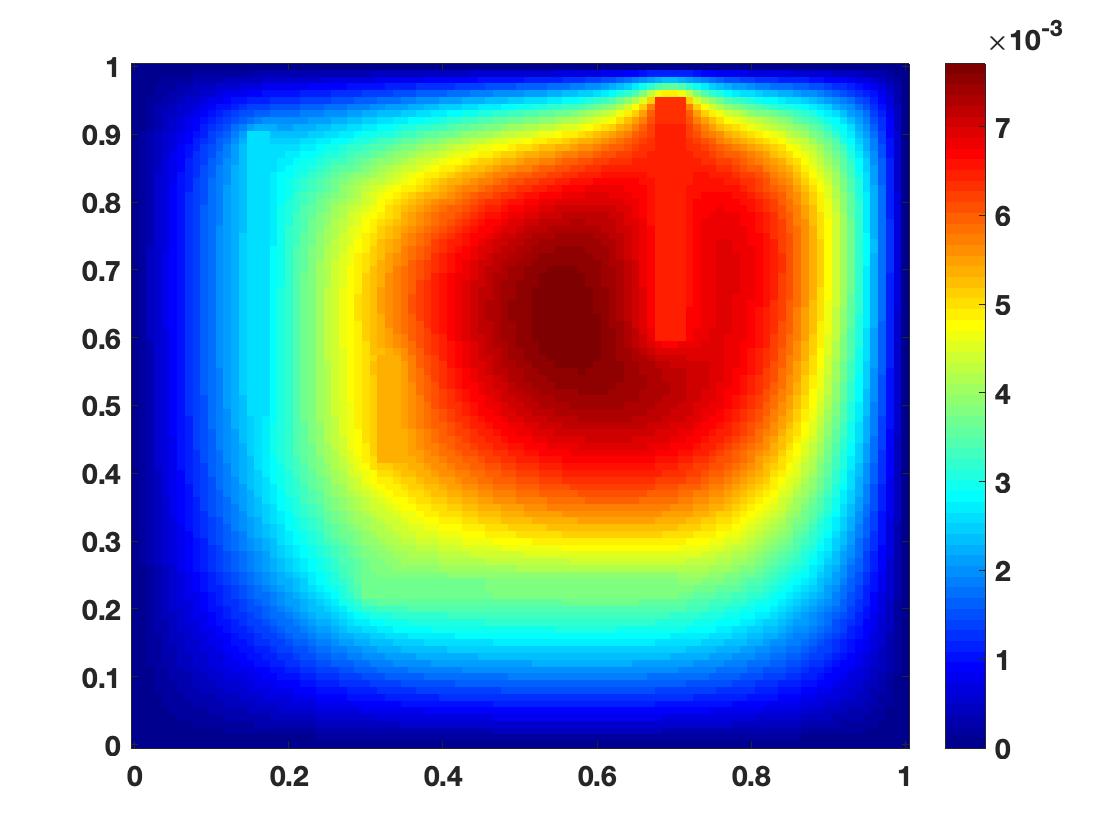}
		\includegraphics[trim={2.9cm 1.6cm 1.0cm 0.6cm},clip,width=0.24 \textwidth]{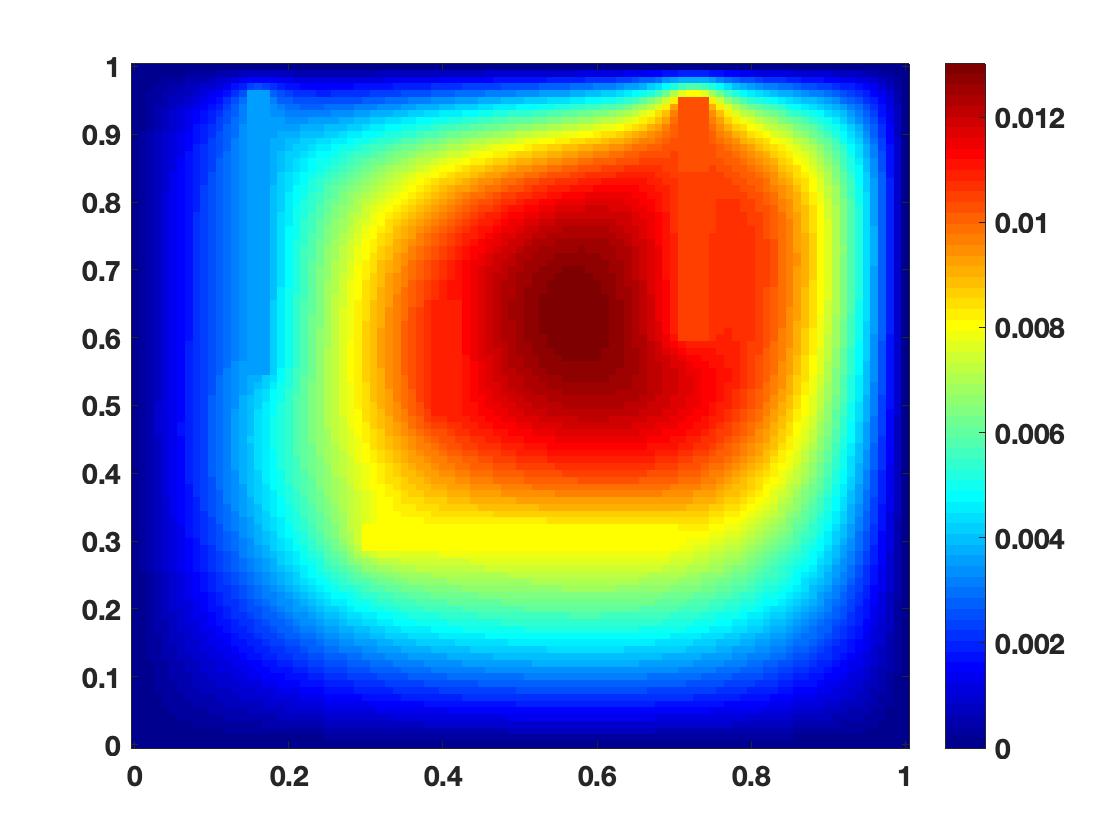}
		\includegraphics[trim={2.9cm 1.6cm 1.0cm 0.6cm},clip,width=0.24 \textwidth]{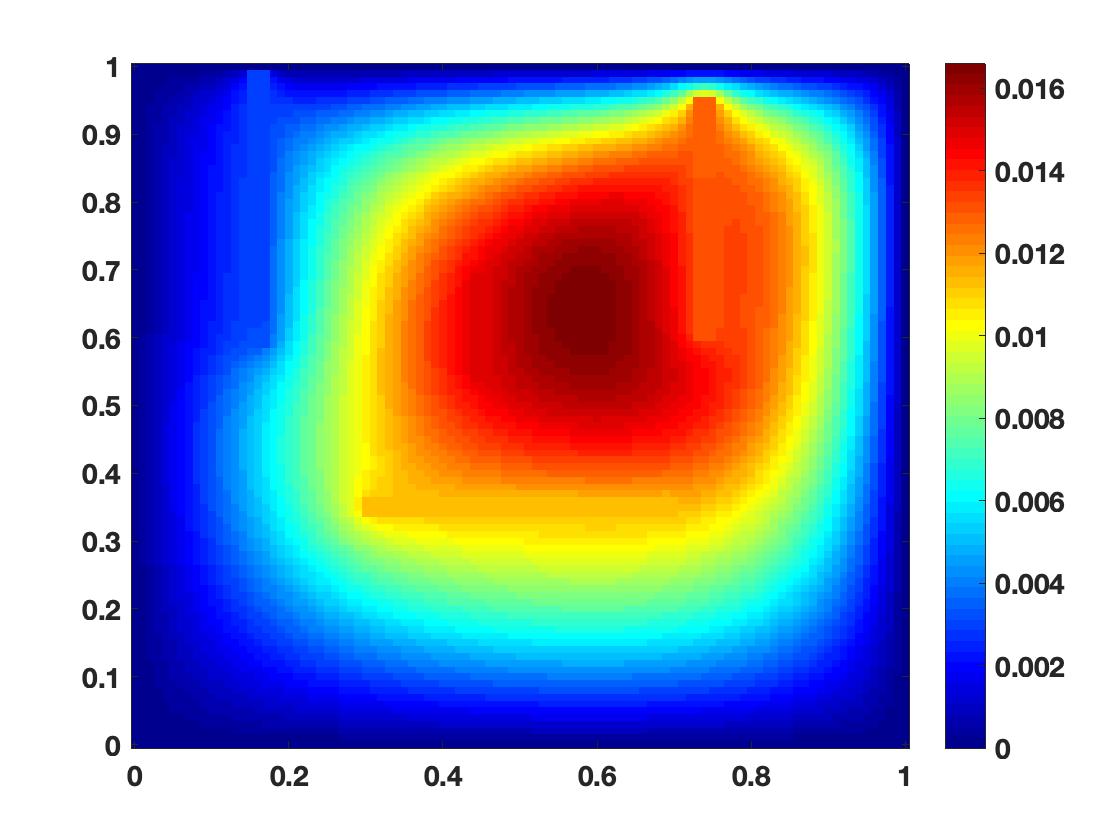}
		\caption{Snapshot of the reference solutions  $\tilde{U}_{h,\delta,t}$ for  $t=0.2,0.5,0.8,1.0$.}
\end{figure}
The coarse spatial and temporal mesh size we use is $H=0.1$ and $\Delta t=0.1$. The number of spatial and temporal oversampling layers $\ell_x$ and $\ell_t$ are chosen to be $\ell=\ell_x=\ell_t \in \{1,2,\cdots,5\}$. We present the snapshot of numerical solutions $\tilde{U}^\ell_{\text{ms},t}$ for $t= 0.2,0.5,0.8,1.0$ with the oversampling layer $\ell=1,2,3$ in Figure \ref{Exp2:multiscale_sol}.
\begin{figure}[H]
		\centering
		\includegraphics[trim={2.9cm 1.6cm 1.0cm 0.6cm},clip,width=0.24 \textwidth]{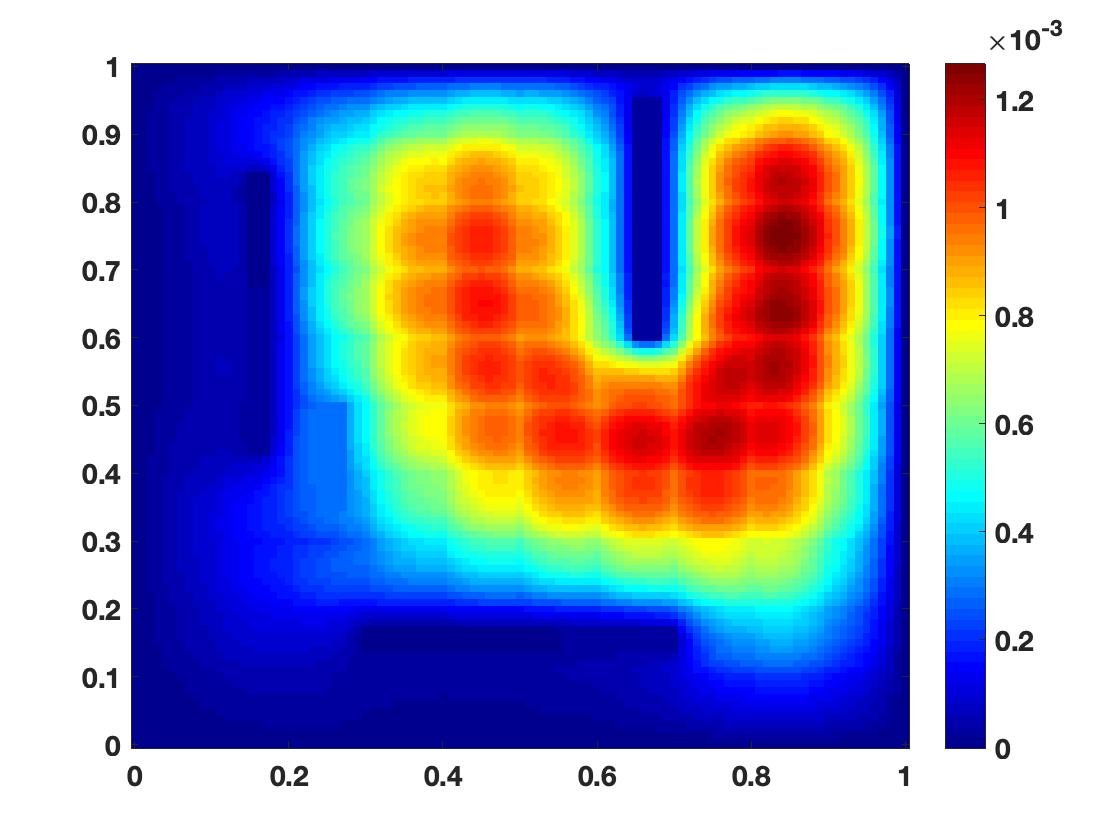}
		\includegraphics[trim={2.9cm 1.6cm 1.0cm 0.6cm},clip,width=0.24 \textwidth]{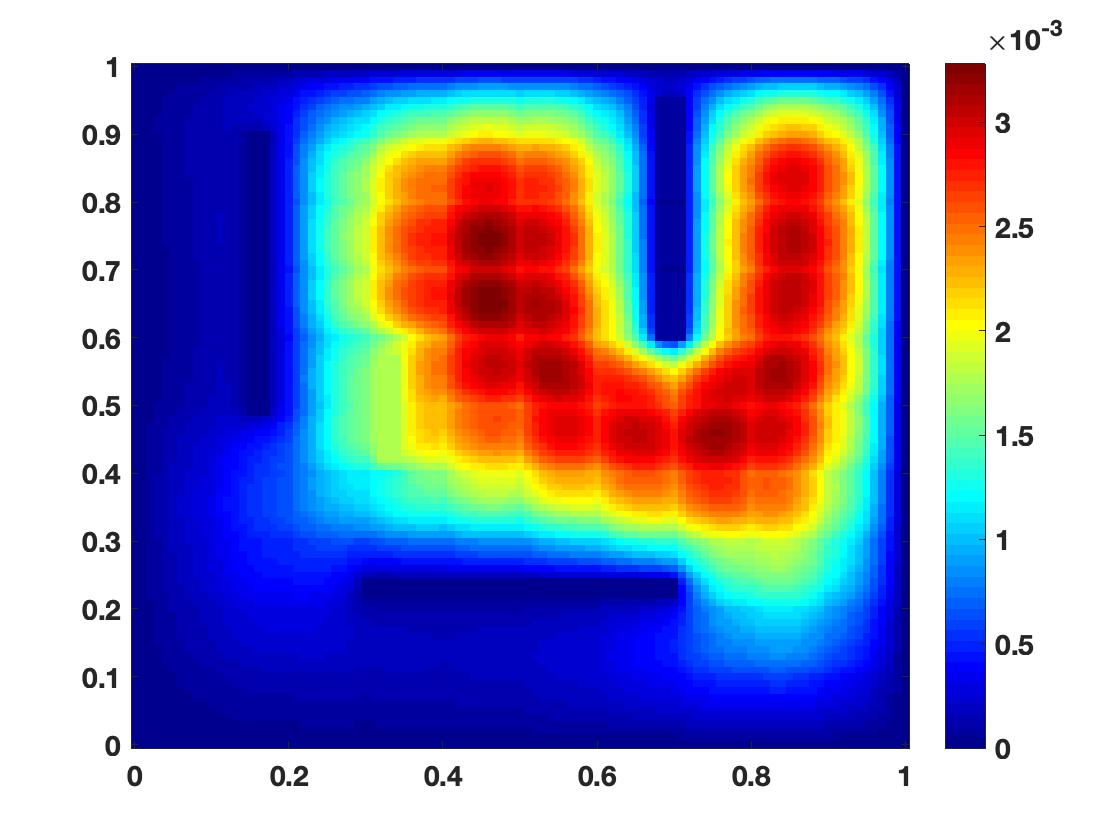}
		\includegraphics[trim={2.9cm 1.6cm 1.0cm 0.6cm},clip,width=0.24 \textwidth]{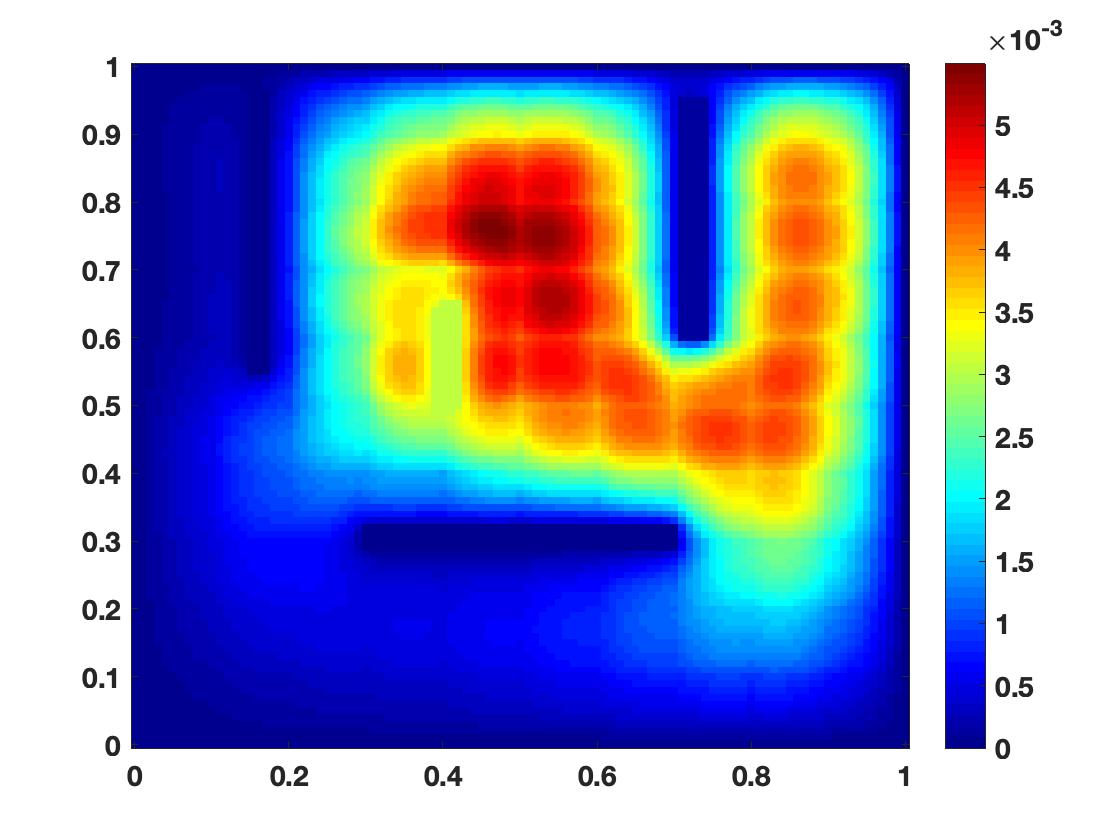}
		\includegraphics[trim={2.9cm 1.6cm 1.0cm 0.6cm},clip,width=0.24 \textwidth]{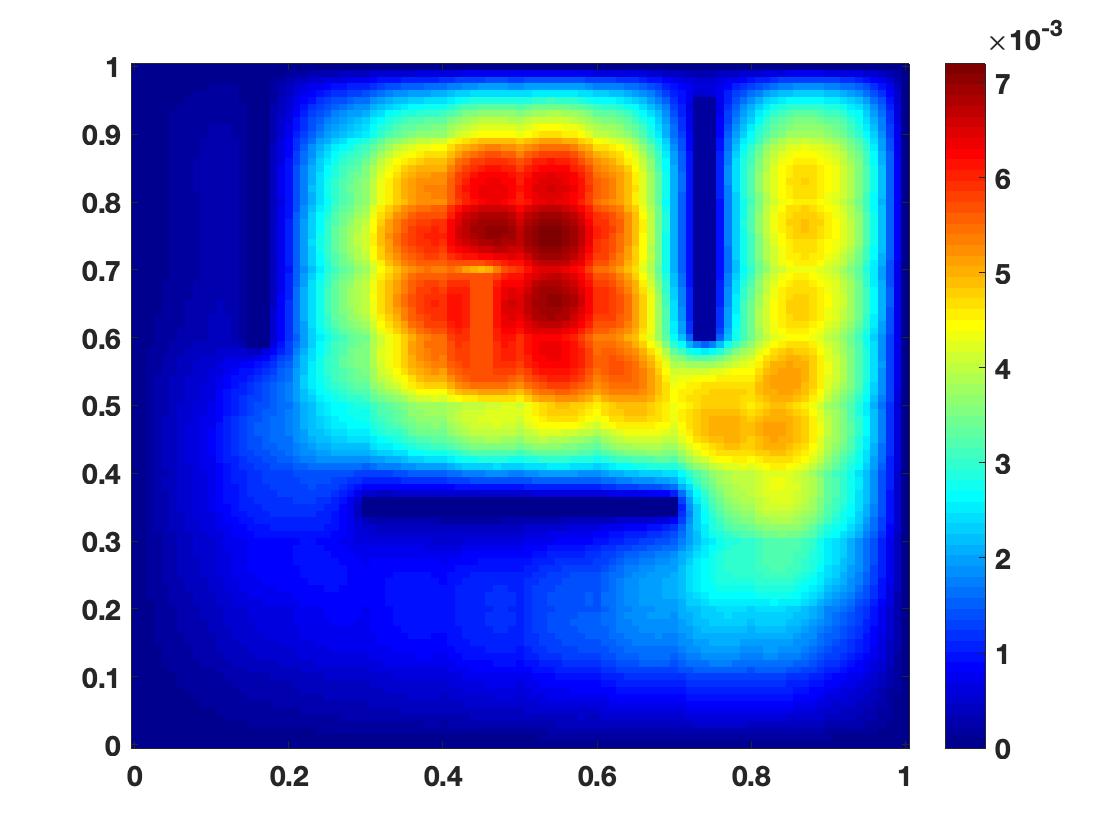}\\
		\includegraphics[trim={2.9cm 1.6cm 1.0cm 0.6cm},clip,width=0.24 \textwidth]{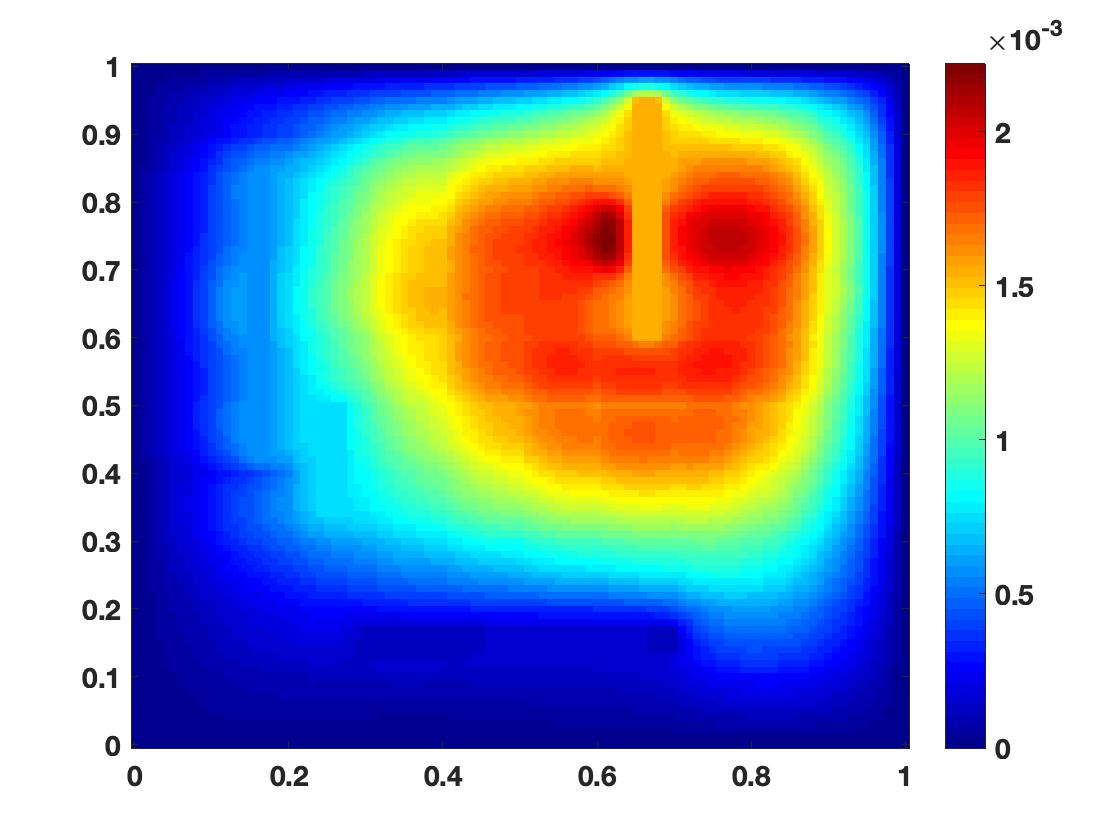}
		\includegraphics[trim={2.9cm 1.6cm 1.0cm 0.6cm},clip,width=0.24 \textwidth]{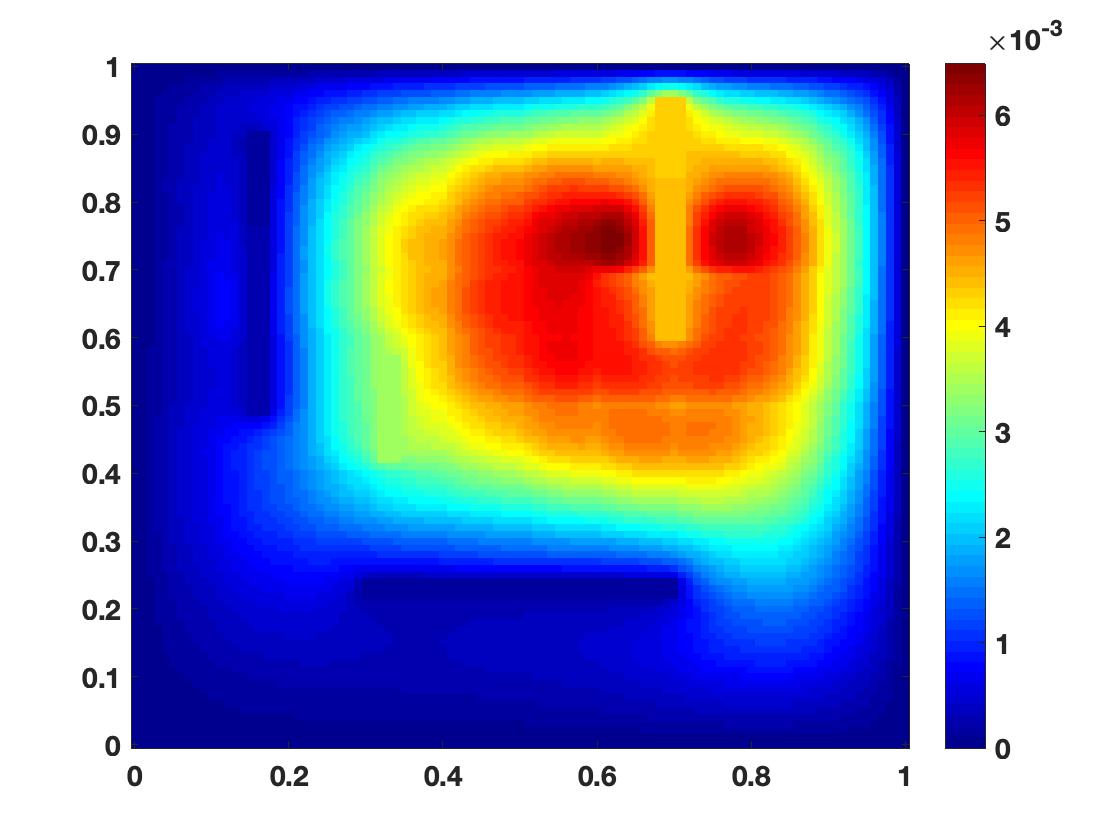}
		\includegraphics[trim={2.9cm 1.6cm 1.0cm 0.6cm},clip,width=0.24 \textwidth]{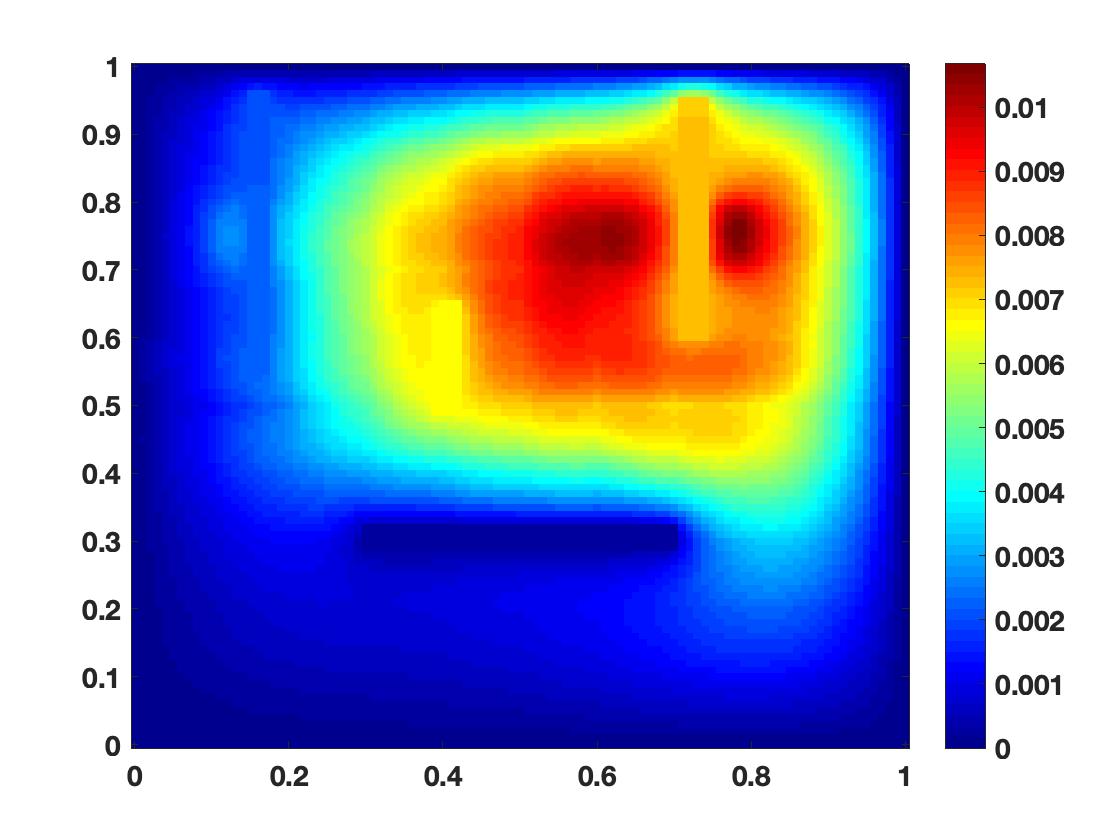}
		\includegraphics[trim={2.9cm 1.6cm 1.0cm 0.6cm},clip,width=0.24 \textwidth]{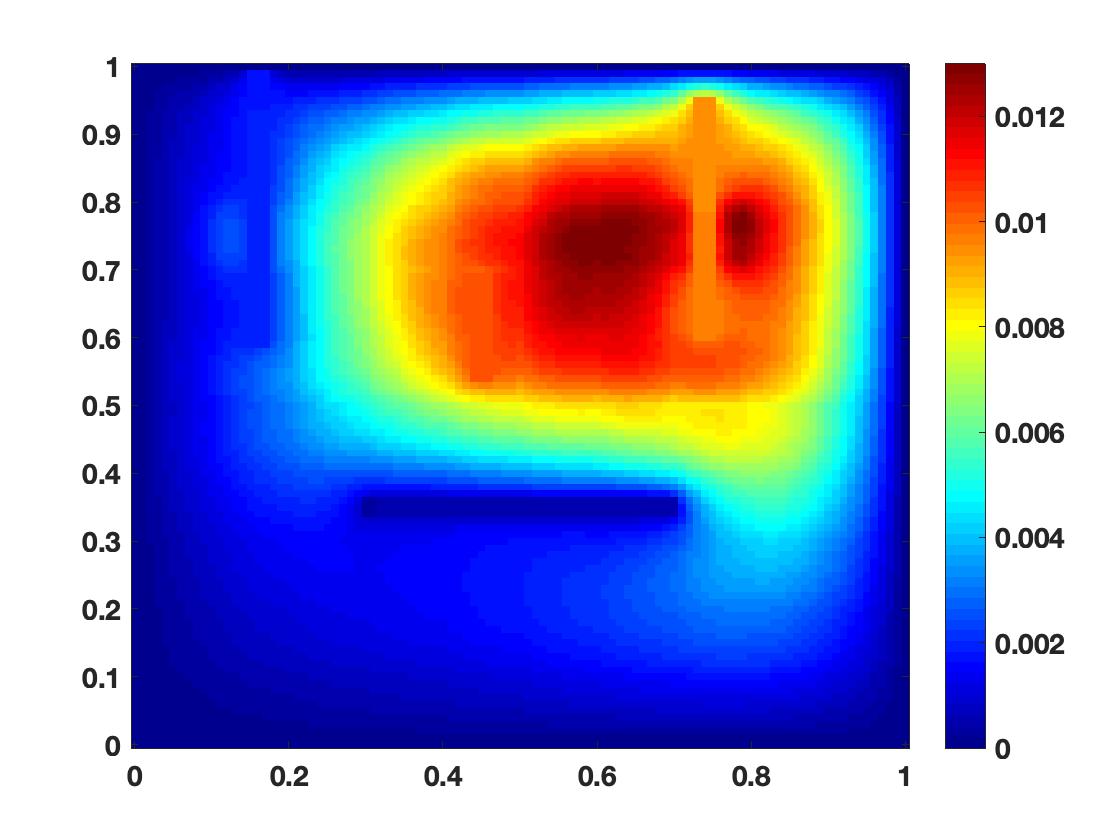}\\
		\includegraphics[trim={2.9cm 1.6cm 1.0cm 0.6cm},clip,width=0.24 \textwidth]{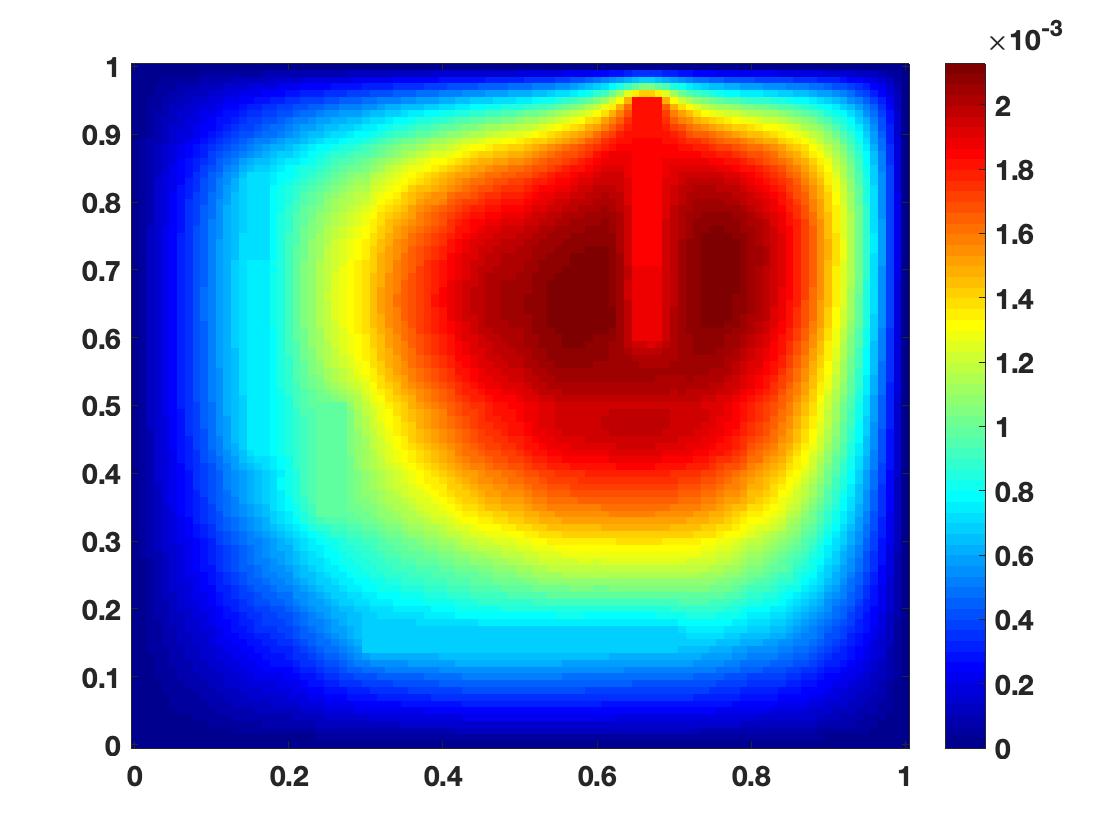}
		\includegraphics[trim={2.9cm 1.6cm 1.0cm 0.6cm},clip,width=0.24 \textwidth]{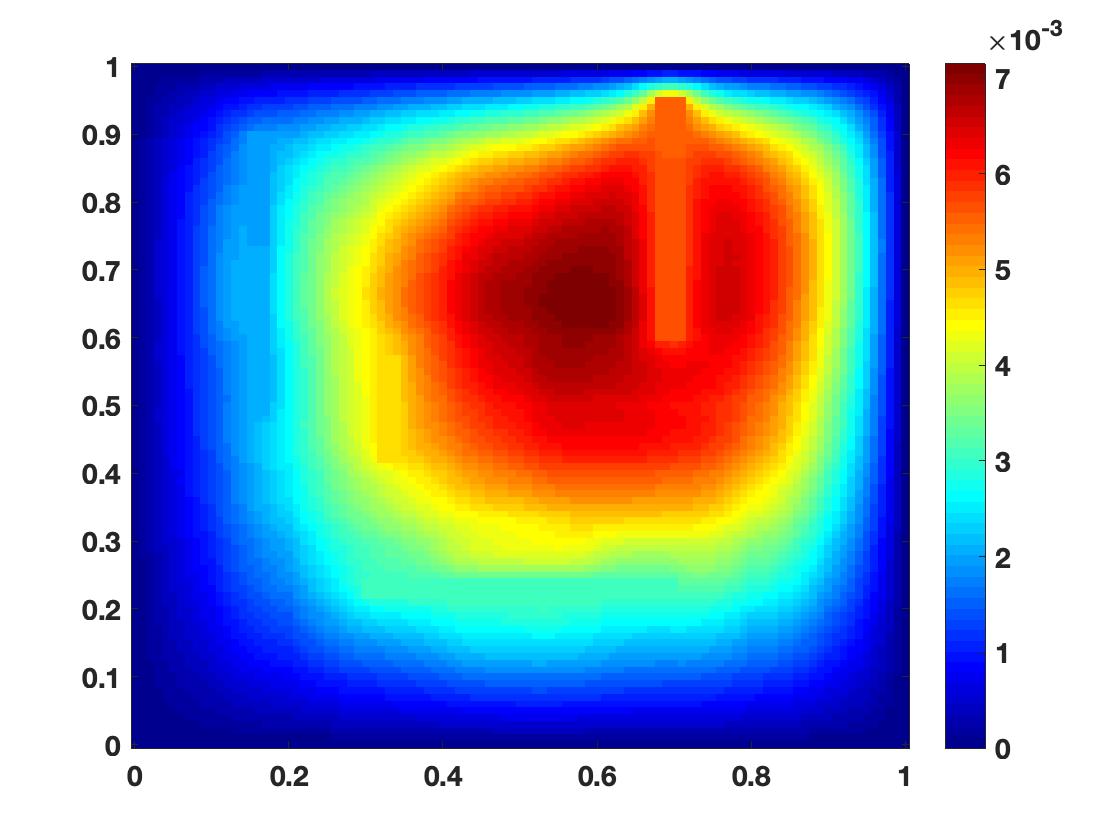}
		\includegraphics[trim={2.9cm 1.6cm 1.0cm 0.6cm},clip,width=0.24 \textwidth]{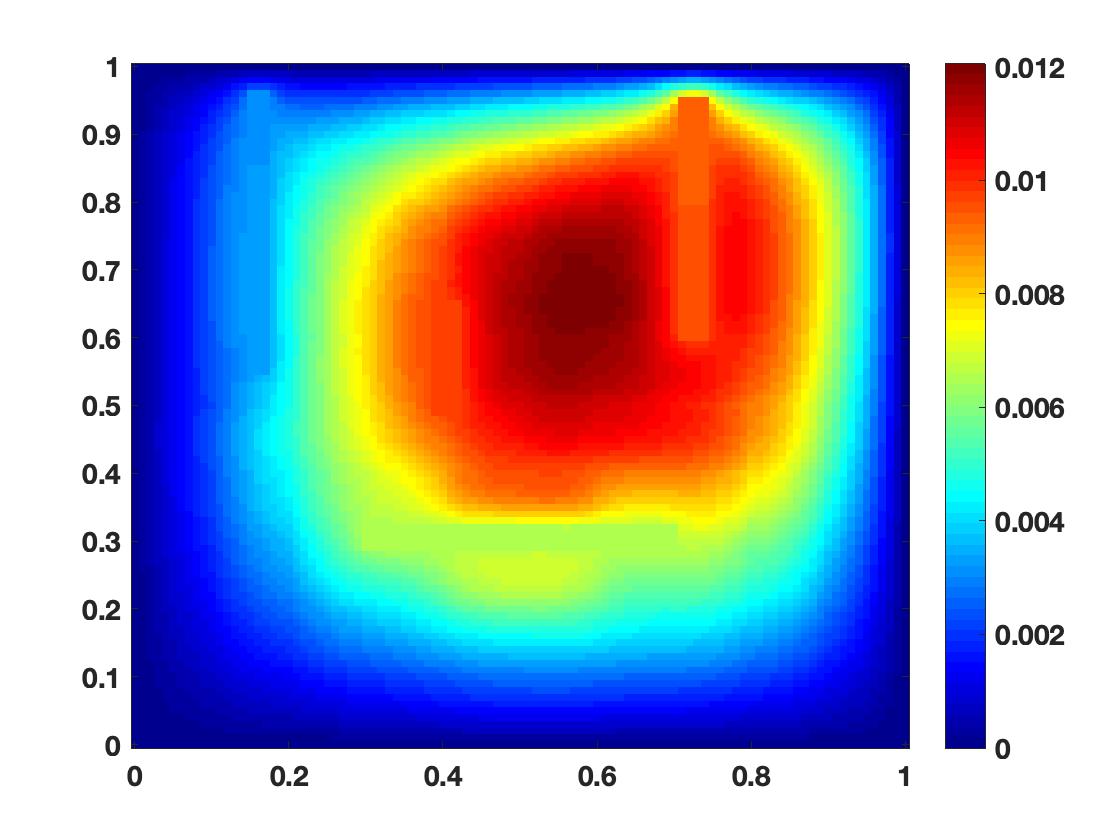}
		\includegraphics[trim={2.9cm 1.6cm 1.0cm 0.6cm},clip,width=0.24 \textwidth]{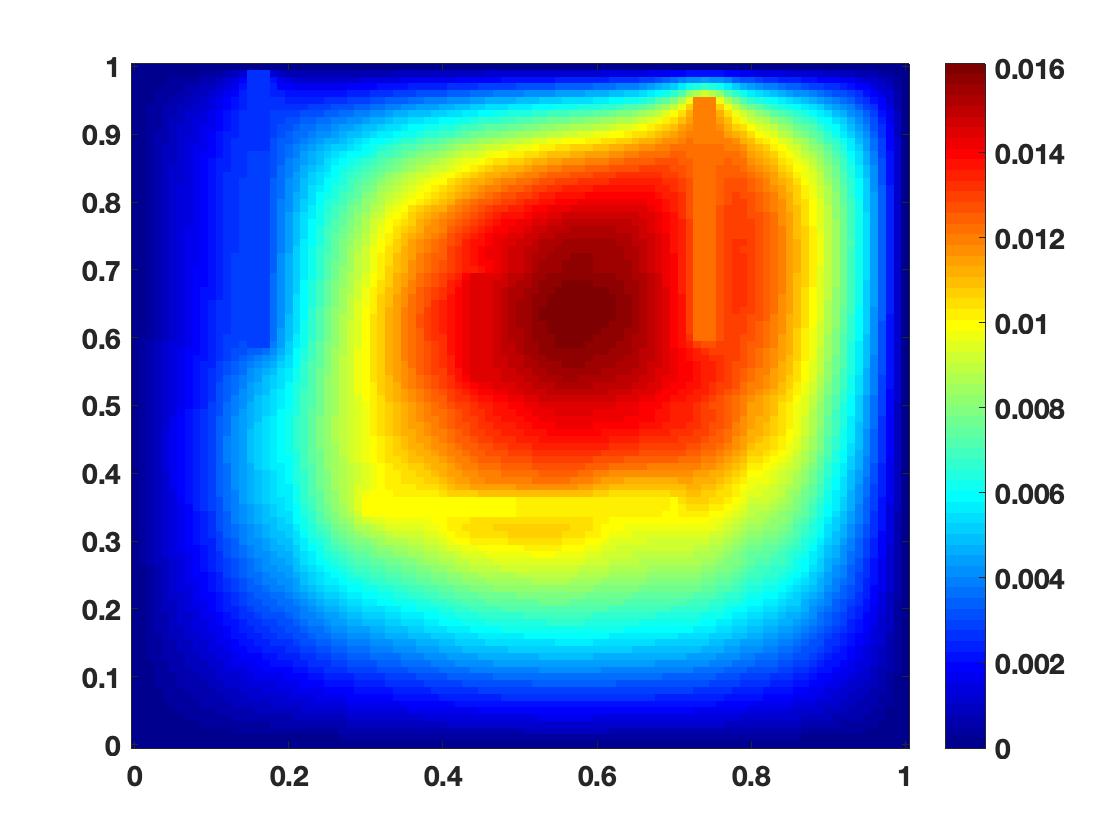}
		\caption{Snapshot of the multiscale solutions $\tilde{U}^\ell_{\text{ms},t}$ at $t=0.2,0.5,0.8,1.0$ with oversampling layer $\ell=1$ (top), $\ell=2$ (middle), $\ell=3$ (bottom). }
		\label{Exp2:multiscale_sol}
\end{figure}

The convergence history in relative $L^2$-norm and relative $H_{\kappa}^1$-norm  with oversampling layers number $\ell=1,2,\cdots,5$ are presented in Table \ref{Table: Exp2_RelativeError}.
\begin{table}[H]
	\begin{center}
	\begin{tabular}{|c|c|c|c|c|c|c|}
	\hline
	$\ell$ & $\text{Rel}_{H_{\kappa}^1}^\ell $&$\text{Rel}_{L^2}^\ell$
	\\ \hline
	 1 & 80.2825& 68.3637\\
 2 &  51.5355& 22.0861\\
 3 &  17.1313& 5.1881\\
 4 &  0.5724 & 0.0658\\
 5& 0.1876 &0.04265 \\
	\hline
	\end{tabular}
	\end{center}
	\vspace{-.4cm}
	\caption{Convergence history of  Experiment 2.}
	\label{Table: Exp2_RelativeError}
	\end{table}

\section{Conclusions}\label{sec:conclusion_SP}
In this paper, we propose an efficient numerical solver for parabolic  problems with moving channelized media.
This approach identifies channels inside each space-time coarse block and defines a piece-wise constant functions as auxiliary functions. The multiscale basis functions are constructed by solving  local problems in the oversampled regions subject to constraints, which guarantee the local multiscale basis functions decay exponentially outside the oversampled regions. We present convergence analysis for the proposed space-time NLMC method. Two numerical experiments are conducted, which show that the proposed approach can provide a good accuracy.


	\bibliographystyle{abbrv}
	\bibliography{reference}
\end{document}